\documentclass[paper=a4,headsepline,DIV13]{scrartcl}

\usepackage[utf8]{inputenc}
\usepackage[T1]{fontenc}
\usepackage{lmodern}
\usepackage{amsmath,amsfonts,amssymb,amsthm,mathtools}
\usepackage{graphicx,subfigure}
\usepackage{braket,dsfont,url,enumerate}
\usepackage{hyperref}
\usepackage{pgfplots}

\usepackage[textsize=small]{todonotes}
\presetkeys{todonotes}{color=red!70}{}

\renewenvironment{abstract}{\minisec{Abstract}}{\par\vspace{.1in}}
\newenvironment{keywords}{\minisec{Key Words}}{\par\vspace{.1in}}
\newenvironment{AMS}{\minisec{AMS subject classification}}{\par\vspace{.1in}}

\theoremstyle{plain}
\newtheorem{theorem}{Theorem}[section]
\newtheorem{corollary}[theorem]{Corollary}
\newtheorem{lemma}[theorem]{Lemma}
\newtheorem{proposition}[theorem]{Proposition}

\theoremstyle{remark}
\newtheorem{remark}[theorem]{Remark}

\numberwithin{equation}{section}

\newcommand{\abs}[1]{\lvert#1\rvert}
\newcommand{\Abs}[1]{\left\lvert#1\right\rvert}
\newcommand{\norm}[1]{\lVert#1\rVert}

\newcommand{\R}{\mathds{R}}
\newcommand{\N}{\mathds{N}}
\newcommand{\HH}{\mathds{H}}

\newcommand{\PP}{\mathcal{P}}
\newcommand{\QQ}{\mathcal{Q}}
\newcommand{\TT}{\mathcal{T}}
\newcommand{\II}{\mathcal{I}}
\newcommand{\NO}{\mathcal{N}_\Om}
\newcommand{\NY}{\mathcal{N}_Y}
\newcommand{\NOY}{\mathcal{N}_{\Om,Y}}

\newcommand{\Om}{\Omega}

\newcommand{\V}{\mathfrak{v}}
\newcommand{\U}{\mathfrak{u}}
\newcommand{\F}{\mathfrak{f}}

\newcommand{\trO}{\operatorname{tr}_\Om}

\renewcommand{\phi}{\varphi}
\renewcommand{\epsilon}{\varepsilon}

\begin{document}
\title{\boldmath{$hp$}-Finite Elements for Fractional Diffusion}

\author{%
Dominik Meidner\footnotemark[2]
\and
Johannes Pfefferer\footnotemark[2]
\and
Klemens Schürholz\footnotemark[2]
\and
Boris Vexler\footnotemark[2]
}

\markright{Meidner, Pfefferer, Schürholz, Vexler: $hp$-Finite Elements for Fractional Diffusion}

\maketitle

\renewcommand{\thefootnote}{\fnsymbol{footnote}}

\footnotetext[2]{Technical University of Munich, Department of Mathematics, Chair of Optimal
Control, Garching / Germany (meidner@ma.tum.de,
pfefferer@ma.tum.de, klemens.schuerholz@tum.de, vexler@ma.tum.de)}

\renewcommand{\thefootnote}{\arabic{footnote}}

\begin{abstract}
	The purpose of this work is to introduce and analyze a numerical scheme to efficiently solve boundary value problems involving the spectral fractional Laplacian. The approach is based on a reformulation of the problem posed on a semi-infinite cylinder in one more spatial dimension. After a suitable truncation of this cylinder, the resulting problem is discretized with linear finite elements in the original domain and with $hp$-finite elements in the extended direction. The proposed approach yields a drastic reduction of the computational complexity in terms of degrees of freedom and even has slightly improved convergence properties compared to a discretization using linear finite elements for both the original domain and the extended direction. The performance of the method is illustrated by numerical experiments.
\end{abstract}




\section{Introduction}
In this work, we are concerned with boundary value problems involving the fractional Laplacian,
being the prototype of a nonlocal operator. To be more specific: Let $\Omega\subset\R^d$ for
$d\in\set{1,2,3}$ be a bounded, convex, polygonal or polyhedral domain.  We are interested in the
solution of the boundary value problem
\begin{equation}\label{eq:Prob}
  \begin{aligned}
    (-\Delta)^s\U&=\F&\quad&\text{in }\Omega,\\
    \U&=0&\quad&\text{on }\partial\Omega,
  \end{aligned}
\end{equation}
where $(-\Delta)^s$ denotes the spectral fractional Laplacian of order $s\in(0,1)$ defined by the
eigenvalues and eigenfunctions of the standard Laplacian, precisely introduced in Section
\ref{sec:ContProb}. The main purpose of this paper is to introduce and analyze a numerical scheme
to efficiently solve problem~\eqref{eq:Prob}.

Our approach is based on the following
equivalent reformulation of problem \eqref{eq:Prob} posed on the semi-infinite cylinder
$\set{(x,y)\in\Omega\times (0,\infty)}\subset \R^{d+1}$: Let $u$ be the weak solution of the
extended problem
\[
  \begin{aligned}
    \operatorname{div}(y^{1-2s}\nabla u)&=0&\quad&\text{in } \Omega\times (0,\infty),\\
    u&=0&&\text{on }\partial\Omega\times[0,\infty),\\
    \partial_{\nu^{1-2s}} u&=d_s \F&&\text{on }\Omega\times\set{0},
  \end{aligned}
\]
see Section \ref{sec:ContProb} for more details.  Then, the trace $\U=u(\cdot,0)$ is the solution
of the fractional boundary value problem \eqref{eq:Prob}.

In contrast to the nonlocal problem \eqref{eq:Prob}, the extended problem is localized. However, a
direct application of a finite element method to the extended problem is not feasible because of the
semi-infinite domain. As remedy, the exponential decay of $u$ in direction $y$ towards infinity (see
Proposition \ref{prop:truncation}) can be employed such that a truncation of the semi-infinite
cylinder to $\Omega\times(0,Y)$ becomes reasonable. The extended problem posed on the truncated
cylinder can be discretized using finite elements. However, due to the degenerate/singular nature of
the extended problem, anisotropic meshes are favorable in order to obtain an optimally convergent
numerical scheme. Moreover, the height $Y$ of the truncated cylinder needs to be chosen dependent on
the mesh parameter to ensure the aforementioned convergence. This approach was already pursued in
\cite{Nochetto2015} using a discretization with first degree tensor product finite elements on graded meshes in the extended direction, see
also~\cite{chen2016multilevel,chen2015pde,nochetto2015convergence,nochetto2016pde} for related
results. If $h_\Om$ denotes the mesh parameter and $\NO$ the number of degrees of freedom in $\Om$
then the approach from \cite{Nochetto2015} yields a discretization error of order $O(h_\Om \abs{\ln h_\Om}^s)$ in the corresponding energy norm associated
with~\eqref{eq:Prob} while solving problems with
$O(\NO^{1+1/d})$ degrees of freedom.

In this work, we introduce and analyze a discretization of the
truncated problem  with linear finite elements in the original domain $\Omega$ and with $hp$-finite elements on a geometric mesh in the extended direction. This drastically reduces the computational complexity to $O(\NO(\ln\NO)^2)$
degrees of freedom and even yields a slightly better convergence rate of order $O(h_\Om)$.
Especially, when $\NO$ is large, the difference between the factor $(\ln\NO)^2$ and $\NO^{1/d}$
becomes clearly perceptible. For instance, in our numerical experiments we could reduce the number
of degrees of freedom by a factor of about $111$ to obtain an error of less than $9\cdot 10^{-3}$ in
the case $s=0.8$, see Section~\ref{sec:NumEx} for more details. We also notice that our approach and
results are not limited to the spectral fractional Laplacian. They naturally extend by only minor
modifications to fractional powers of general second order elliptic operators.

Let us briefly give an overview on other numerical approaches from the literature to solve boundary
value problems involving the fractional Laplacian: Due to the spectral definition of the operator, it seems to be
natural to compute an approximating, discrete spectral decomposition of the standard Laplacian in
order to get an approximation of the solution of \eqref{eq:Prob}, see
\cite{MR2252038,MR2300467,MR2800568}. However, this may result in solving a large number of discrete
eigenvalue problems. Another approach to determine an approximation to the solution of problem
\eqref{eq:Prob} is analyzed in \cite{bonito2015numerical1}, see also
\cite{bonito2017approximation,bonito2017numerical,bonito2015numerical2} for related results. In that
reference, $(-\Delta)^{-s}$ is represented in terms of Bochner integrals involving
$(I-t^2\Delta)^{-1}$ for $t\in(0,\infty)$. Subsequently, different quadrature formulas to
approximate this integral are analyzed which require multiple evaluations of
$(I-t_i^2\Delta_h)^{-1}$ with $t_i$ being a quadrature point and $-\Delta_h$ denoting a finite
element discretization to $-\Delta$. Numerical approaches for the integral definition of the
fractional Laplacian, which is not equivalent to the spectral definition considered in the present
paper, can be found in \cite{acosta2016short, acosta2017finite, acosta2017fractional,
acosta2016regularity, borthagaray2016finite, d2013fractional, guan2015theta, guan2017analysis,
huang2014numerical}.

This paper is organized as follows: In Section \ref{sec:ContProb}, we state the definition of the
fractional Laplacian, formulate the extended problem in detail, and introduce the functional
framework needed for the subsequent error analysis. Moreover, in this section, we are concerned with
several properties of the solution of the extended problem such as a series representation and
corresponding regularity results. The discrete, extended problem posed on the truncated cylinder is formulated at the beginning of Section
\ref{sec:Disc}. In the extended direction, we distinguish between graded meshes and $h$-FEM, and geometric meshes and $hp$-FEM, see Sections \ref{sec:lin_def} and \ref{sec:p_def}. The error analysis is given in Section \ref{sec:ErrorEst}. Thereby, in Section~\ref{sec:lin}, we mainly recover the results of~\cite{Nochetto2015}. The
reason for doing this is twofold. First, we are able to slightly improve the mesh grading condition used in~\cite{Nochetto2015}. However, the main
reason to analyze the $h$-FEM on graded meshes before developing the analysis for the $hp$-method considered in Section~\ref{sec:p} is, that the techniques we use are almost identical for both cases, but the details are simpler for $h$-FEM. Implementation aspects and numerical experiments, which underline the efficiency of our approach, are presented in Section \ref{sec:NumEx}. In the appendix, we collect different results for special functions defined by the modified Bessel functions of second kind. These are especially needed in Section \ref{sec:ContProb} for the discussion of the solution of the extended problem.

Finally, we notice that, in the following, $c$ denotes a generic constant which will always be
independent of the mesh
parameter $h_\Omega$ when we analyze the discretization error.

\section{Continuous Problem}\label{sec:ContProb}

Let $-\Delta$ be the $L^2(\Omega)$ realization of the Laplacian with homogeneous Dirichlet boundary
conditions.  It is well-known that $-\Delta$ has a compact resolvent and its eigenvalues form a
non-decreasing sequence $0<\lambda_1\le\lambda_2\le\cdots\le\lambda_k\le\cdots$ satisfying
$\lim_{k\to\infty}\lambda_k=\infty$. We denote by $\varphi_k$ the  orthonormal $L^2(\Om)$
eigenfunctions associated with $\lambda_k$ fulfilling
\[
  \int_{\Omega} \nabla \varphi_{k} \cdot \nabla v\,dx=\lambda_k\int_{\Omega}\varphi_{k}v\,dx \quad \forall v\in H^1_0(\Omega).
\]
For any $s\geq0$, we introduce the fractional order Sobolev space
\[
  \HH^s(\Omega)=\Set{\V\in L^2(\Omega) | \norm{\V}_{\HH^s(\Omega)}^2=\sum_{k=1}^\infty
  \lambda_k^s\V_k^2<\infty \text{ with } \V_k=\int_{\Omega}\V\varphi_k\,dx}.
\]
Moreover, we denote by $\HH^{-s}(\Om)$ the dual space of $\HH^s(\Om)$. Then, the spectral fractional
Laplacian is defined for $s\in(0,1)$ on the space $\HH^s(\Omega)$ as the limit
\[
(-\Delta)^s\U=\sum_{k=1}^\infty\lambda_k^s\U_k\varphi_k\in\HH^{-s}(\Om)\quad \text{with}\quad
\U_k=\int_{\Omega}\U\varphi_k\,dx.
\]
Due to the Cauchy-criterion the limit exists for any $\U\in \HH^s(\Om)$. Thus, problem \eqref{eq:Prob} has to be understood as: Given $\F\in\HH^s(\Omega)$, find $\U\in \HH^s(\Om)$ such that
\begin{equation}\label{eq:Probweak}
	\sum_{k=1}^\infty\lambda_k^s\U_k\V_k=\int_\Om \F\V\,dx\quad \forall \V\in\HH^s(\Om)\quad \text{with}\quad
	\V_k=\int_{\Omega}\V\varphi_k\,dx.
\end{equation}
\begin{proposition}\label{prop:SolUF}
  For any $\F\in\HH^{-s}(\Om)$, problem~\eqref{eq:Prob} admits a unique solution $\U\in\HH^s(\Om)$
  fulfilling $\norm{\U}_{\HH^s(\Om)}=\norm{\F}_{\HH^{-s}(\Om)}$. Moreover, there is the series representation
  \[
    \U=\sum_{k=1}^\infty\U_k\phi_k \quad\text{with}\quad \U_k=\lambda_k^{-s}\F_k\quad\text{and}\quad \F_k=\int_{\Omega}\F\varphi_k\,dx.
  \]
\end{proposition}
\begin{proof}
	The existence of an unique solution and the equality of the norms is a consequence of the Riesz representation theorem. The series representation of $\U$ is obtained by testing \eqref{eq:Probweak} with $\varphi_{m}\in \HH^s(\Om)$ and using the orthogonality of the eigenfunctions.
\end{proof}
\begin{remark}
  Due to the definition of the fractional Laplacian and the previous result, we observe that problem
  \eqref{eq:Prob} is already meaningful without additionally imposing the homogeneous Dirichlet
  boundary conditions since these are already included in the definition of the operator. Moreover,
  we notice that the regularity of $\U$ can be described in classical fractional Sobolev spaces as
  well, since
  \[
    \HH^s(\Omega)=
    \begin{cases}
      H^s(\Omega),&\text{for } 0<s<\frac 12,\\
      H_{00}^{\frac 12}(\Omega),&\text{for } s=\frac 12,\\
      H_0^s(\Omega),&\text{for }\frac 12<s<1.
    \end{cases}
  \]
  For more details we refer to, e.g.,~\cite{Nochetto2015}.
\end{remark}

Problem \eqref{eq:Prob} can equivalently be posed on a semi-infinite cylinder. In $\R^d$, this is
due to Caffarelli and Silvestre \cite{Caffarelli2007}. The restriction to bounded domains $\Omega$
was considered by Stinga and Torrea in \cite{Stinga2010}, see also \cite{Cabre2010,Capella2011}.
This kind of extension is the basis for the computational approaches in the subsequent sections.

In order to state the extended problem, we first introduce the required notation. We denote by
$C=\Omega \times (0,\infty)$ the aforementioned semi-infinite cylinder and by $\partial_LC =
\partial\Omega \times [0,\infty)$ its lateral boundary. We also need to define a truncated cylinder:
for $Y > 0$, the truncated cylinder is given by $C_Y=\Om\times(0,Y)$ with its lateral boundary
$\partial_LC_Y = \partial\Omega \times [0,Y]$. As $C$ and $C_Y$ are objects in $\R^{d+1}$, we use
$y$ to denote the extended variable, such that a vector $(x,y) \in \R^{d+1}$ admits the
representation $(x,y) = (x_1, x_2,\dots, x_d, y)$. Similarly, the gradient in $\R^{d+1}$ has the
representation
$\nabla=(\nabla_x,\partial_y)=(\partial_{x_1},\partial_{x_2},\dots,\partial_{x_d},\partial_y)$.

Next we introduce weighted Sobolev spaces with a weight function $y^\alpha$ for $\alpha \in (-1,1)$.
In this regard, let $D\subset \R^d\times(0,\infty)$ be an open set, such as $C$ or $C_Y$. Then, we
define the weighted space $L^2(D,y^\alpha)$ as the space of all measurable functions on $D$ with
finite norm $\norm{v}_{L^2(D,y^\alpha)}=\norm{y^{\frac\alpha2}v}_{L^2(D)}$. Similarly, the space
$H^1(D,y^\alpha)$ denotes the space of all functions $v\in L^2(D,y^\alpha)$ whose weak derivatives
of first order belong to $L^2(D,y^\alpha)$.

To study the extended problems, we introduce the space
\[
  \mathring{H}^1_L(C,y^\alpha)=\Set{v\in H^1(y^\alpha,C) | v=0\text{ on } \partial_LC}.
\]
The space $\mathring{H}^1_L(C_Y,y^\alpha)$ is defined analogously, but endowed with zero Dirichlet
boundary conditions also on $\Om\times\set{Y}$:
\[
  \mathring{H}^1_L(C_Y,y^\alpha)=\Set{v\in H^1(y^\alpha,C_Y) | v=0\text{ on }
  \partial_LC_Y\cup (\Omega\times \set{Y})}.
\]

For $v\in \mathring{H}^1_L(C,y^\alpha)$, we denote by $\trO v$ the trace of $v$ onto
$\Om\times\set{0}$, i.e., $\trO v=v(\cdot ,0)$. 
\begin{proposition}\label{prop:trace}
  For $\alpha=1-2s$, it holds
  \[
    \trO \mathring{H}^1_L(C,y^\alpha)=\HH^s(\Om)\quad\text{and}\quad
    \norm{\trO v}_{\HH^s(\Om)}\le c\norm{v}_{\mathring{H}^1_L(C,y^\alpha)}.
  \]
\end{proposition}
\begin{proof}
  See~\cite[Proposition 1.8]{Cabre2010} for $s=\frac12$ and~\cite[Proposition 2.1]{Capella2011} for
  $s\neq\frac12$.
\end{proof}

Now, we are able to state the extended problem: Given $\F\in \HH^{-s}(\Omega)$, find $u\in
\mathring{H}^1_L(C,y^\alpha)$ such that
\begin{equation}\label{eq:weakExtension}
  \int_{C}y^\alpha \nabla u \cdot \nabla v \, d(x,y)= d_s\langle \F, \trO v\rangle_{\HH^{-s}(\Om),\HH^s(\Om)} \quad
  \forall v\in \mathring{H}^1_L(C,y^\alpha)
\end{equation}
with $\alpha=1-2s$ and $d_s=2^\alpha \frac{\Gamma(1-s)}{\Gamma(s)}$.  That is, the function $u\in
\mathring{H}^1_L(C,y^\alpha)$ is a weak solution of
\[
  \begin{aligned}
    \operatorname{div}(y^\alpha\nabla u)&=0&\quad&\text{in } C,\\
    u&=0&&\text{on }\partial_L C,\\
    \partial_{\nu^\alpha} u&=d_s \F&&\text{on }\Omega\times\set{0},
  \end{aligned}
\]
where we have set $\partial_{\nu^\alpha} u(x,0)=\lim_{y\to 0}y^\alpha \partial_yu(x,y)$. Note that
subsequently, the parameter $\alpha$ will always be equal to $1-2s$ for the considered $s\in(0,1)$.

In the remainder of this section, we discuss several properties of the solution $u$ to
\eqref{eq:weakExtension}.  Due to the following proposition, it is reasonable to determine the
solution $u$ of \eqref{eq:weakExtension} in order to get the solution $\U$ of \eqref{eq:Prob}.
\begin{proposition}\label{prop:equi}
  For $\F\in\HH^{-s}(\Om)$, the extended problem~\eqref{eq:weakExtension} admits a unique solution $u\in
  \mathring{H}^1_L(C,y^\alpha)$. Furthermore, $\U=\trO u\in\HH^s(\Om)$ solves~\eqref{eq:Prob}.
\end{proposition}

\begin{proof}
  See~\cite[Lemma 2.1]{Capella2011}.
\end{proof}

We have the following regularity result.
\begin{proposition}\label{prop:regU}
  Let $u\in \mathring{H}^1_L(C,y^\alpha)$ be the solution of~\eqref{eq:weakExtension} and
  $\F\in\HH^{1-s}(\Om)$. Then, it holds
  \[
    \norm{y^{\frac\alpha2}\nabla_x^2u}_{L^2(C)}^2+\norm{y^{\frac\alpha2}\partial_y\nabla_xu}_{L^2(C)}^2\le
    c\norm{\F}_{\HH^{1-s}(\Om)}^2.
  \]
\end{proposition}

\begin{proof}
  \cite[Theorem 2.7]{Nochetto2015} yields
  \[
    \norm{y^{\frac\alpha2}\Delta_xu}_{L^2(C)}^2+\norm{y^{\frac\alpha2}\partial_y\nabla_xu}_{L^2(C)}^2=
    d_s\norm{\F}_{\HH^{1-s}(\Om)}^2
  \]
  with $d_s$ given above. Convexity of $\Om$ then implies the assertion.
\end{proof}

For a series expansion of the solution to~\eqref{eq:weakExtension}, we introduce the function
$\psi_s\colon[0,\infty)\to\R$ given by
\begin{equation}\label{eq:psi}
  \psi_s(z)= c_s z^s K_s(z) \qquad\text{with}\qquad c_s=\frac{2^{1-s}}{\Gamma(s)}.
\end{equation}
Here, $K_s(z)$ denotes the modified Bessel function of second kind, see, e.g.,~\cite[Section
9.6]{Abramowitz1964} and $\Gamma(s)$ denotes the gamma function.

\begin{proposition}\label{prop:series}
  For $\F\in\HH^{-s}(\Om)$, let $\U\in\HH^s(\Om)$ be the solution of~\eqref{eq:Prob} and $u\in \mathring{H}^1_L(C,y^\alpha)$
  be the solution of~\eqref{eq:weakExtension}. Then, it holds
  \[
    u(x,y)=\sum_{k=1}^\infty \U_k\phi_k(x)\psi_{s,k}(y),
  \]
  where $\U_k=\int_\Om\U\phi_k\,dx$ and $\psi_{s,k}(y)=\psi_s(\sqrt{\lambda_k}y)$.
\end{proposition}

\begin{proof}
  For $s\neq\frac12$, see~\cite[Proposition 2.1]{Capella2011}. For $s=\frac12$, the result was
  proved with $\psi_{\frac12}(z)=e^{-z}$ in~\cite[Proposition 2.2]{Cabre2010}. However,
  by~\cite[9.6.23]{Abramowitz1964} and the relation $\Gamma(\frac12)=\sqrt{\pi}$, it holds
  \[
    c_{\frac12}z^{\frac12}K_{\frac12}(z)=e^{-z}=\psi_{\frac12}(z).
  \]
  Hence, $\psi_s(z)=c_s z^s K_s(z)$ holds also in the case $s=\frac12$.
\end{proof}

We state an estimate of the derivatives of $\psi_{s,k}$ which we will use later. The result is based
on properties of $\psi_s$ which are analyzed in the appendix.

\begin{corollary}\label{cor:psi-2}
  Let $r\in[0,1]$. There exists a constant $c > 0$ depending only on $s$, such that for any $y>0$
  and $n\in\N$ it holds
  \[
    \abs{y^n\psi_{s,k}^{(n)}(y)}\le c8^n  n!\lambda_k^{s-\frac r2}y^{2s-r}.
  \]
\end{corollary}

\begin{proof}
  Simple calculations yield
  \[
    \psi_{s,k}^{(n)}(y) = \frac{d^n}{d y^n} \psi(\sqrt{\lambda_k} y)
    = (\sqrt{\lambda_k})^n\psi^{(n)}(\sqrt{\lambda_k} y) .
  \]
  Consequently, we obtain by means of Lemma~\ref{lem:psi-2}
  \[
    \abs{y^n \psi_{s,k}^{(n)}(y)} =
    \abs{(\sqrt{\lambda_k} y)^n\psi^{(n)}(\sqrt{\lambda_k} y)} \leq
    c8^n n!(\sqrt{\lambda_k} y)^{2s-r}=c8^n n!\lambda_k^{s-\frac r2}y^{2s-r}.\qedhere
  \]
\end{proof}

Next, we state a result about the exponential decay of $\psi_{s,k}$ and its derivative. It is based
on corresponding results for $\psi_s$ and its derivative, proved in the appendix.

\begin{corollary}\label{cor:psi-3}
  Let $y\geq1$, $r_1\geq \min(s,\frac12)-s$ and $r_2\geq \min(1-s,\frac12)-s$. Then, there exists a
  constant $c_1$ only depending on $r_1$, $s$, and $\lambda_1$ and a constant $c_2$ only depending on
  $r_2$, $s$, and $\lambda_1$ such that
  \[
    \abs{y^{r_1}\psi_{s,k}(y)}\leq c_1 \lambda_k^{-\frac{r_1}2}e^{-\frac{\sqrt{\lambda_k}}2y}\quad
    \text{and}\quad \abs{y^{r_2}\psi_{s,k}'(y)}\leq c_2 \lambda_k^{-\frac{r_2}2}
    e^{-\frac{\sqrt{\lambda_k}}2y}.
  \]
\end{corollary}
\begin{proof}
  Let $s_0=\min(s,\frac12)$. According to \cite[Theorem 5]{MillerSamko}, we get for $z\geq0$ that
  $z^{s_0}e^zK_s(z)$ is a decreasing function. Consequently, we obtain
  \begin{equation}\label{eq:monotonelambda1}
    (\sqrt{\lambda_k})^{s_0}e^{\sqrt{\lambda_k}}K_s(\sqrt{\lambda_k})\leq (\sqrt{\lambda_1})^{s_0}e^{\sqrt{\lambda_1}}K_s(\sqrt{\lambda_1})
  \end{equation}
  since the sequence $(\lambda_k)_{k\in\N}$ is non-decreasing. Moreover,
  due to the definition of $\psi_{s,k}$, we deduce
  \[
    \abs{y^{r_1}\psi_{s,k}(y)}= \lambda_k^{-\frac{r_1}2}\abs{(\sqrt{\lambda_k}y)^{r_1}\psi_{s}(\sqrt{\lambda_k}y)}.
  \]
  Thus, by setting $a=\sqrt{\lambda_k}$ and $z=\sqrt{\lambda_k}y\ge a$ in Lemma \ref{lem:psi-3}
  (\ref{it:1}), we obtain the validity of the first inequality of the assertion by means of
  \eqref{eq:monotonelambda1}. The second inequality can be deduced in the same manner employing
  Lemma \ref{lem:psi-3} (\ref{it:2}).
\end{proof}

As already mentioned, for computational reasons, the semi-infinite cylinder $C$ will be truncated to
$C_Y=\Om\times(0,Y)$ for some $Y>0$ later on. Because of this, the behavior of $\nabla u$ for
$y\to\infty$ will play a role. It can be estimated as follows:

\begin{proposition}\label{prop:truncation}
  For $\F\in\HH^{-s}(\Om)$, let $u\in\mathring{H}^1_L(C,y^\alpha)$ be the solution
  of~\eqref{eq:weakExtension}. Then, there exists a constant $c>0$ such that for every $Y\ge 1$, it
  holds
  \[
    \norm{y^{\frac\alpha2}\nabla u}_{L^2(C\setminus C_Y)}\le c
    e^{-\frac{\sqrt{\lambda_1}}2Y}\norm{\F}_{\HH^{-s}(\Om)}.
  \]
\end{proposition}

\begin{proof}
  The result can be found in \cite[Proposition 3.1]{Nochetto2015}. For the sake of completeness, we
  state a (slightly different) proof here. According to Proposition \ref{prop:series}, we obtain by
  using the definition of the eigenfunctions $\varphi_k$ and its orthogonality
  \begin{align*}
    \norm{y^{\frac\alpha2}\nabla u}_{L^2(C\setminus C_Y)}^2
    &=\int_{Y}^\infty y^\alpha\int_\Omega\bigl\{\abs{\nabla_xu(x,y)}^2+\abs{\partial_yu(x,y)}^2\bigr\}\,dx\,dy\\
    &=\sum_{k=1}^\infty \U_k^2\int_{Y}^\infty \bigl\{\lambda_k\abs{y^{\frac\alpha2}\psi_{s,k}(y)}^2+ \abs{y^{\frac\alpha2}\psi_{s,k}'(y)}^2\bigr\}\,dy\\
    &\leq c\sum_{k=1}^\infty \U_k^2\int_{Y}^\infty
    \bigl\{\lambda_k^{s+\frac12}e^{-\sqrt{\lambda_k}y}+\lambda_k^{s-\frac12}
    e^{-\sqrt{\lambda_k}y}\bigr\}\,dy,
  \end{align*}
  where we employed Corollary \ref{cor:psi-3} in the last step with
  $r_1=r_2=\frac{\alpha}{2}=\frac12-s$. Calculating the integral and using that
  $(\lambda_k)_{k\in\N}$ is a non-decreasing sequence, yields
  \[
    \norm{y^{\frac\alpha2}\nabla u}_{L^2(C\setminus C_Y)}^2\leq c\sum_{k=1}^\infty \U_k^2
    \bigl\{\lambda_k^{s}e^{-\sqrt{\lambda_k}Y}+\lambda_k^{s-1} e^{-\sqrt{\lambda_k}Y}\bigr\}\leq
    ce^{-\sqrt{\lambda_1}Y}\sum_{k=1}^\infty \lambda_k^{s}\U_k^2 =
    ce^{-\sqrt{\lambda_1}Y}\norm{\F}^2_{\HH^{-s}(\Om)},
  \]
  where we used Proposition \ref{prop:SolUF} in the last step.
\end{proof}

\section{Discretization}\label{sec:Disc}

Let $\TT_\Om$ be a conforming and quasi-uniform triangulation of $\Omega$ which is admissible in the
sense of Ciarlet. For each $\TT_\Om=\set{K}$, let $K\subset \R^d$  be an element that is
isoparametrically equivalent either to the unit cube or to the unit simplex in $\R^d$.  We introduce
the global mesh parameter $h_\Omega$ with respect to the triangulation of $\Omega$ by
$h_\Omega=\max_{K\in\TT_\Omega}\operatorname{diam}K$. We always assume that $h_\Omega\leq \frac12$.
On $\TT_\Om$, we define a finite element space $V_h$ as
\[
  V_h=\Set{\V\in C^0(\overline{\Om}) | \V\bigr\rvert_{K}\in\mathds{P}_1(K),~K\in \TT_\Om,~\V\bigr\rvert_{\partial\Om}=0}.
\]
In case that $K$ is a simplex then $\mathds{P}_1(K)=\PP_1(K)$, the set of polynomials on the element
$K$ of degree at most $1$. If $K$ is a cube then $\mathds{P}_1(K)$ equals $\QQ_1(K)$, the set of
polynomials on $K$ of degree at most 1 in each variable.  The number of degrees of freedom in $V_h$
is denoted by $\NO$. It holds $\NO =O(h_\Om^{-d})$.

Furthermore, let $\II_Y=\set{I_m}$ be a triangulation of the interval $(0,Y)$ in the sense
that $[0,Y]=\bigcup_{m=1}^{M}I_m$ with $I_m=[y_{m-1},y_{m}]$ and $M\in\N$ exactly specified below in
the Sections~\ref{sec:lin_def} and~\ref{sec:p_def}. Moreover, let $h_m=\abs{I_m}$.  Next,
we introduce a polynomial degree vector $p=(p_1,p_2,\dots,p_M)\in\N^M$ which will assign to each
element $I_m\in \II_Y$ a maximal polynomial degree $p_m$. It will be exactly specified in
the Sections~\ref{sec:lin_def} and~\ref{sec:p_def}. On $\II_Y$, we define the finite element space
$V_M$ as
\[
  V_M=\Set{\xi\in C^0(\overline I) | \xi\bigr\rvert_{I_m}\in\PP_{p_m}(I_m),~I_m\in \II_Y,~\xi(Y)=0},
\]
where $\PP_{p_m}(I_m)$ denotes the space of polynomials up to degree $p_m$ on $I_m$.

Now, the triangulations $\TT_{\Om,Y}$ of the cylinder $C_Y$ are constructed as tensor product
triangulations by means of $\TT_\Om$ and $\II_Y$, i.e., $\TT_{\Om,Y}=\set{T}$ with $T=K\times I_m$
for $K\in\TT_\Om$ and $I_m\in\II_Y$.  By means of the previous considerations, we define the finite
element space $V_{h,M}$ posed on the tensor product mesh $\TT_{\Om,Y}$ by
\[
  V_{h,M}=V_h\otimes V_M=\operatorname{span}\Set{v | v(x,y)=\V(x)\xi(y),~\V\in V_h,~\xi\in
  V_M}\subset\mathring{H}^1_L(C_Y,y^\alpha).
\]
Note, that each function $v_h\in V_{h,M}$ vanishes on the lateral boundary of $C_Y$ and on its top.
As a consequence, the extension by zero of $v_h$ to the semi-infinite cylinder $C$ belongs to
$\mathring{H}^1_L(C,y^\alpha)$. Without further mention, we consider this type of extension for each
$v_h\in V_{h,M}$ whenever needed.

With the just introduced notation, we define approximations to the solution $u$ of
\eqref{eq:weakExtension} as follows: Find $u_h\in V_{h,M}$ satisfying
\begin{equation}\label{eq:weakdiscreteExtension}
  \int_{C_Y}y^\alpha \nabla u_h \cdot \nabla v_h \, d(x,y)= d_s \langle \F, \trO
  v_h\rangle_{\HH^{-s}(\Om),\HH^s(\Om)} \quad \forall v_h\in V_{h,M},
\end{equation}
where we recall that $\alpha=1-2s$ and $d_s=2^\alpha \tfrac{\Gamma(1-s)}{\Gamma(s)}$. Note that
$\operatorname{tr}_\Omega u_h=u_h(\cdot,0)$ will be used as an approximation of $\U$.

We distinguish two possible types of discretization in the artificial $y$ direction, which will be
defined in the following two sections.

\subsection{Graded meshes and \boldmath{$h$}-FEM}\label{sec:lin_def}

In this section, let $M\in\N$ to be determined later. We set
\[
  y_m=\biggl(\frac{m}{M}\biggr)^{\frac1\mu} Y\quad \text{for}\quad
  m=0,1,\ldots,M\quad \text{and}\quad p=(1,1,\ldots,1)\in \N^M,
\]
where $\mu\in(0,1]$ represents the grading parameter in direction $y$. Hence, for $\mu<1$ the
triangulation $\TT_{\Om,Y}$ is anisotropic. Moreover, due to the choice of the polynomial degree
vector $p$, the discrete space $V_{h,M}$ consists of globally continuous and piecewise multilinear
functions on $C_Y$. 

We start with a result regarding the diameter $h_m$ of the elements $I_m$.

\begin{lemma}\label{lemma:gradedmesh}
  It holds
  \[
    h_1= M^{-\frac1\mu}Y\quad\text{and}\quad
    \frac{2^{\frac{\mu-1}\mu}}{\mu}y_m^{1-\mu}Y^\mu M^{-1}\leq h_m\leq \frac{1}{\mu} y_m^{1-\mu}Y^\mu
    M^{-1}\quad\text{for } m=2,3,\ldots,M.
  \]
\end{lemma}

\begin{proof}
  The first equality is obvious due to the definition of $y_1$. For the second, we observe that
  there exists a $m_*\in(m-1,m)$ such that
  \[
    h_m=y_m-y_{m-1}=(m^{\frac1\mu}-(m-1)^{\frac1\mu})Y M^{-\frac1\mu}=\frac1\mu m_*^{\frac1\mu-1}Y
    M^{-\frac1\mu}=\frac1\mu \biggl(\biggl(\frac{m_*}{M}\biggr)^{\frac{1}\mu}Y\biggr)^{1-\mu} Y^\mu
    M^{-1}
  \]
  due to the mean value theorem. Since $\frac12m\leq m-1\leq m$ for $m\ge 2$, the result follows.
\end{proof}

Since we consider a discretization with linear polynomials in direction $y$, the number of degrees
of freedom $\NY$ in $V_M$ is proportional to $M$, i.e., it holds $\NY=O(M)$.

\subsection{Geometric meshes and \boldmath{$hp$}-FEM}\label{sec:p_def}

For the second possibility of discretization presented here, the mesh in direction $y$ is chosen
as a geometric mesh. That is, for chosen $\sigma\in(0,1)$ and $M\in\N$ to be determined later, the
nodes $y_0,y_1,\dots,y_M$ are given by
\begin{equation}\label{eq:geo_mesh}
  y_0=0 \qquad \text{and} \qquad y_m=\sigma^{M-m} Y\quad \text{for}\quad m=1,2,\dots,M.
\end{equation}
Further, we define $p=(p_1,p_2,\dots,p_M)\in\N^M$ to be a linear degree vector with slope $\beta >
0$. That is, there exists a constant $c \ge 1$ such that for all $m=1,2,\dots,M$ the following
relation is fulfilled:
\begin{equation}\label{eq:lin_deg_vec}
  1 + \beta \ln\frac{h_m}{h_1}\leq p_m \leq 1 + c\beta  \ln\frac{h_m}{h_1}.
\end{equation}

We start with collecting some basic results for the discretization considered in this subsection.

\begin{lemma}\label{lem:GeoMesh_geoMeshProps}
  For a geometric mesh given by~\eqref{eq:geo_mesh}, there holds for $h_m=\abs{I_m}$
  \[
    \begin{aligned}
      h_1 &= y_1 = \sigma^{M-1}Y,\\
      h_m &= (1-\sigma) y_m = (\sigma^{-1} - 1) y_{m-1}&\quad &\text{for }  m = 2,3,\dots,M,\\
      h_m &= (1-\sigma)\sigma^{1-m}h_1&&\text{for }  m = 2,3,\dots,M.
    \end{aligned}
  \]
\end{lemma}

\begin{proof}
  The first equation of the assertion is obvious due to the definition of the geometric mesh. For
  $m=2,3,\dots,M$, we obtain
  \[
    h_m = y_m-y_{m-1} = (\sigma^{M-m} - \sigma^{M-m+1})Y =
    \begin{cases}
      \sigma^{M-m}(1 - \sigma)Y= (1-\sigma) y_m,\\
      \sigma^{M-m+1}(\sigma^{-1} - 1)Y= (\sigma^{-1} - 1) y_{m-1}.
    \end{cases}
  \]
  As a consequence, we get
  \[
    \frac{h_m}{h_1}= \frac{\sigma^{M-m}(1 - \sigma)Y}{\sigma^{M-1}Y} =
    (1-\sigma)\sigma^{1-m}.\qedhere
  \]
\end{proof}

\begin{lemma}\label{lem:GeoMesh_linDegDepOnK}
  For a geometric mesh given by~\eqref{eq:geo_mesh} and a linear degree vector $p\in\N^M$ in the
  sense of~\eqref{eq:lin_deg_vec}, it holds $p_1 =1$ and there is a constant $c\ge 1$ such that
  \[
    1 + \beta  \big(\ln(1-\sigma)+(1-m)\ln\sigma\big)
    \leq p_m\leq 1 + \beta c \big(\ln(1-\sigma)+(1-m)\ln\sigma\big)
  \]
  for all $m=2,3,\dots,M$.
\end{lemma}

\begin{proof}
  Since $\ln\frac{h_1}{h_1}=0$, there obviously holds $p_1=1$. According to
  Lemma~\ref{lem:GeoMesh_geoMeshProps}, we get for $m=2,3\dots,M-1$
  \[
    \ln\frac{h_m}{h_1} = \ln\left( (1-\sigma)\sigma^{1-m} \right) =  \ln(1-\sigma)+(1-m)\ln\sigma,
  \]
  which shows the second assertion.
\end{proof}

\begin{lemma}\label{lem:GeoMesh_dofsInY}
  Let $\II_Y$ be a geometric triangulation of $(0,Y)$ given by~\eqref{eq:geo_mesh} and let
  $p\in\N^M$ be a linear degree vector as defined in~\eqref{eq:lin_deg_vec}. Then, for the number of
  degrees of freedom $\NY$ in $V_M$, it holds
  \[
    \NY = 1+\sum_{m=1}^M p_m =O(M^2).
  \]
\end{lemma}

\begin{proof}
  On each interval $I_m$ we locally have $p_m+1$ degrees of freedom. Moreover, two neighboring
  intervals always share one degree of freedom. Thus, we have
  \[
    \NY =\sum_{m=1}^M (p_m +1) - (M-1) =1+\sum_{m=1}^M p_m,
  \]
  which proves the assertion having in mind Lemma \ref{lem:GeoMesh_linDegDepOnK}.
\end{proof}

\section{Error Estimates}\label{sec:ErrorEst}

We start with providing a general error estimate between the solutions $u\in
\mathring{H}^1_L(C,y^\alpha)$ of~\eqref{eq:weakExtension} and $u_h\in V_{h,M}$
of~\eqref{eq:weakdiscreteExtension} in weighted norms.

\begin{lemma}\label{lem:bestapprox}
  Let  $u\in \mathring{H}^1_L(C,y^\alpha)$ be the solution of~\eqref{eq:weakExtension} and $u_h\in
  V_{h,M}$ be the solution of~\eqref{eq:weakdiscreteExtension}. Then, it holds
  \[
    \norm{y^{\frac\alpha2}\nabla(u-u_h)}_{L^2(C)}\le \min_{v_h\in V_{h,M}}\norm{y^{\frac\alpha2}\nabla(
    u-v_h)}_{L^2(C_Y)}+\norm{y^{\frac\alpha2}\nabla u}_{L^2(C\setminus C_Y)}.
  \]
\end{lemma}

\begin{proof}
  Testing \eqref{eq:weakExtension} with $v_h\in V_{h,M}$ and then subtracting
  \eqref{eq:weakdiscreteExtension} from this equation yields
  \begin{equation}\label{eq:GalerkinO}
    \int_{C_Y}y^\alpha \nabla(u-u_h)\cdot \nabla v_h \,d(x,y)=0.
  \end{equation}
  Based on this, we deduce for all $v_h\in V_{h,M}$
  \begin{align*}
    \norm{y^{\frac\alpha2}\nabla(u-u_h)}_{L^2(C)}^2
    &=\int_{C_Y} y^\alpha \nabla(u-u_h)\cdot\nabla(u-v_h)\,d(x,y)+\int_{C\setminus C_Y}y^\alpha\nabla (u-u_h)\cdot \nabla u\,d(x,y)\\
    &\leq
    \norm{y^{\frac\alpha2}\nabla(u-u_h)}_{L^2(C)}\bigl\{\norm{y^{\frac\alpha2}\nabla(u-v_h)}_{L^2(C_Y)}+\norm{y^{\frac\alpha2}\nabla
    u}_{L^2(C\setminus C_Y)}\bigr\}.
  \end{align*}
  Dividing by $\norm{y^{\frac\alpha2}\nabla(u-u_h)}_{L^2(C)}$ ends the proof.
\end{proof}

Whereas the last term in the estimate of Lemma~\ref{lem:bestapprox} can be treated by means of
Proposition~\ref{prop:truncation}, we have to estimate the first term. To this end, we introduce the
following approximation operators separately for the $x$ and $y$ variables:

By $\pi_x\colon L^2(\Om)\to V_h$, we denote the $L^2$ projection with respect to the $x$ variable on
$\Omega$. For the interpolation with respect to the $y$ direction, we consider each interval $I_m$
separately. For fixed $m\in\set{1,2,\dots,M}$, let $q\in\N$ and $y_{m-1}=x_0,x_1,\ldots,x_q=y_m$
be the Gauss-Lobatto points in $I_m$ and let $l_{i,q}$ denote the corresponding Lagrange polynomials
of order $q$. Then, we define the Gauss-Lobatto interpolant $i_q\colon
C^0(\overline{I_m})\to\PP_q(I_m)$ by
\[
  (i_q\xi)(y)=\sum_{i=0}^q\xi(x_i)l_{i,q}(y).
\]
Further, we define an interpolant $\tilde i_q$ which admits $(\tilde i_q\xi)(y_m)=0$ given by
\[
  (\tilde i_q\xi)(y)=\sum_{i=0}^{q-1}\xi(x_i)l_{i,q}(y).
\]
Based on this, we define the interpolation $i_y^p\colon C^0((0,Y])\to V_M$ for a linear
degree vector $p=(p_1,p_2,\dots,p_M)\in\N^M$ by
\begin{equation}\label{eq:def_i}
  (i_y^p\xi)\bigr\rvert_{I_m}=
  \begin{cases}
    \xi(y_1)&\text{on }I_1,\\
    i_{p_m}\xi\bigr\rvert_{I_m}&\text{on }I_m,~m=2,3,\dots,M-1,\\
    \tilde i_{p_M}\xi\bigr\rvert_{I_M}&\text{on }I_M.
  \end{cases}
\end{equation}
In particular, it holds that $i_y^p\xi$ is constant on $I_1$ and $(i_y^p\xi)(Y)=0$.
For any function $v\in \mathring{H}^1_L(C,y^\alpha)$, we set
\[
  (\pi_x v)(\cdot,y)=\pi_x v(\cdot,y)\quad\text{for a.a.\ }y\in (0,\infty).
\]
Moreover, for the solution $u$ of~\eqref{eq:weakExtension}, the application of $i_y^p$ is defined as
\begin{equation}\label{def:ip}
  (i_y^pu)(x,\cdot)=\sum_{k=1}^\infty \U_k\phi_k(x)i_y^p\psi_{s,k}(\cdot)\quad\text{for a.a.\ } x\in\Om,
\end{equation}
which is well-defined because $\psi_{s,k}\in C^0(\R_+)$.  Thus, by construction, we have
$\pi_xi_y^pu\in V_{h,M}$.

For the first term on the right-hand side of the estimate in Lemma~\ref{lem:bestapprox}, we have the
following result.

\begin{lemma}\label{lem:generalint}
  Let  $u\in \mathring{H}^1_L(C,y^\alpha)$ be the solution of~\eqref{eq:weakExtension}. Then, it
  holds
  \[
    \min_{v_h\in V_{h,M}}\norm{y^{\frac\alpha2}\nabla( u-v_h)}_{L^2(C_Y)}\le
    \norm{y^{\frac\alpha2}\nabla(u-\pi_x u)}_{L^2(C_Y)}+
    c\norm{y^{\frac\alpha2}\nabla(u-i_y^pu)}_{L^2(C_Y)}.
  \]
\end{lemma}

\begin{proof}
  First, we set $v_h=\pi_xi_y^pu\in V_{h,M}$. Then, by introducing $\pi_x u$ as an intermediate
  function, we deduce
  \[
    \min_{v_h\in V_{h,M}}\norm{y^{\frac\alpha2}\nabla( u-v_h)}_{L^2(C_Y)}\le
    \norm{y^{\frac\alpha2}\nabla(u-\pi_x u)}_{L^2(C_Y)}+
    \norm{y^{\frac\alpha2}\nabla\pi_x(u-i_y^pu)}_{L^2(C_Y)}.
  \]
  According to the definition of $\pi_x$, we have $\partial_y \pi_x (u-i_y^pu)=\pi_x\partial_y
  (u-i_y^pu)$ almost everywhere in $C_Y$. As a consequence, we deduce by well known stability
  estimates for the $L^2$ projection $\pi_x$
  \begin{align*}
    \norm{y^{\frac\alpha2}\nabla\pi_x(u-i_y^pu)}_{L^2(C_Y)}&=\int_{0}^{Y}y^\alpha
    \int_\Omega\bigl\{\abs{\nabla_x\pi_x(u-i_y^pu)}^2+\abs{\partial_y\pi_x(u-i_y^pu)}^2\bigr\}\,dx\,dy\\
    &=\int_{0}^{Y}y^\alpha
    \int_\Omega\bigl\{\abs{\nabla_x\pi_x(u-i_y^pu)}^2+\abs{\pi_x\partial_y (u-i_y^pu)}^2\bigr\}\,dx\,dy\\
    &\leq c \int_{0}^{Y}y^\alpha
    \int_\Omega\bigl\{\abs{\nabla_x(u-i_y^pu)}^2+\abs{\partial_y (u-i_y^pu)}^2\bigr\}\,dx\,dy,
  \end{align*}
  which shows the assertion.
\end{proof}

\begin{lemma}\label{lem:Clement}
  For $\F\in\HH^{1-s}(\Om)$, let  $u\in \mathring{H}^1_L(C,y^\alpha)$ be the solution
  of~\eqref{eq:weakExtension}. Then, it holds
  \[
    \norm{y^{\frac\alpha2}\nabla(u-\pi_x u)}_{L^2(C_Y)} \le ch_\Omega \norm{\F}_{\HH^{1-s}(\Om)}.
  \]
\end{lemma}

\begin{proof}
  Analogously to the foregoing proof, by classical estimates for the $L^2$ projection
  $\pi_x$, we obtain
  \begin{align*}
    \norm{y^{\frac\alpha2}\nabla(u-\pi_x u)}_{L^2(C_Y)}^2&=\int_{0}^{Y}y^\alpha
    \int_\Omega\bigl\{\abs{\nabla_x(u-\pi_xu)}^2+\abs{\partial_y(u-\pi_xu)}^2\bigr\}\,dx\,dy\\
    &=\int_{0}^{Y}y^\alpha
    \int_\Omega\bigl\{\abs{\nabla_x(u-\pi_xu)}^2+\abs{\partial_yu-\pi_x(\partial_yu)}^2\bigr\}\,dx\,dy\\
    &\leq ch_\Omega^2\int_{0}^{Y}y^\alpha
    \int_\Omega\bigl\{\abs{\nabla^2_xu}^2+\abs{\nabla_x\partial_yu}^2\bigr\}\,dx\,dy\\
    &=c h_\Om^2\bigl\{
  \norm{y^{\frac\alpha2}\nabla_x^2u}_{L^2(C_Y)}^2+\norm{y^{\frac\alpha2}\partial_y\nabla_xu}_{L^2(C_Y)}^2\bigr\}.
\end{align*}
Then, Proposition~\ref{prop:regU} yields the assertion.
\end{proof}

Next, we are concerned with estimates for the second term in Lemma~\ref{lem:generalint}. To this
end, we consider the interpolation error on each subinterval $I_m\in\II_Y$. Employing the
decomposition from Proposition~\ref{prop:series}  and the definition of the eigenfunctions
$\varphi_{k}$, we obtain
\begin{equation}\label{eq:inty}
  \begin{aligned}
    \norm{y^{\frac\alpha2}\nabla(u-i_y^pu)}_{L^2(\Omega\times I_m)}^2&=\int_{I_m}y^\alpha
    \int_\Omega\bigl\{\abs{\nabla_x(u-i_y^pu)}^2+\abs{\partial_y(u-i_y^pu)}^2\bigr\}\,dx\,dy\\
    &=\int_{I_m}y^\alpha\sum_{k=1}^\infty \U_k^2 \bigl\{\lambda_k(\psi_{s,k}-i_y^p\psi_{s,k})^2+
  ((\psi_{s,k}-i_y^p\psi_{s,k})')^2\bigr\}\,dy\\
  &=\sum_{k=1}^\infty
  \U_k^2\bigl\{\lambda_k\norm{y^{\frac\alpha2}(\psi_{s,k}-i_y^p\psi_{s,k})}^2_{L^2(I_m)} +
  \norm{y^{\frac\alpha2}(\psi_{s,k}-i_y^p\psi_{s,k})'}^2_{L^2(I_m)}\bigr\}.
\end{aligned}
\end{equation}
By using this identity in the following two subsections, we will estimate the terms
\[
  \norm{y^{\frac\alpha2}\nabla(u-i_y^pu)}_{L^2(\Omega\times I_m)}^2,\quad i=1,2,\dots,M
\]
for the two types of triangulations $\II_Y$ and polynomial spaces $V_M$ introduced in the
Section~\ref{sec:lin_def} and~\ref{sec:p_def}.

Thereby, in Section~\ref{sec:lin}, we will mainly recover the results of~\cite{Nochetto2015}. The
reason for doing this is twofold. First, we are able to slightly improve the grading condition from
$\mu<\frac23 s$ (in our notation) of~\cite[Section~5.2]{Nochetto2015} to $\mu<s$. However, the main
reason to analyze the $h$-FEM on graded meshes before developing the analysis for the considered
$hp$-method is, that the techniques we use are almost identical for both cases, but the details are
simpler for $h$-FEM, of course.

Later, in Section~\ref{sec:p}, we will analyze the $hp$-method introduced in
Section~\ref{sec:p_def}, which yields a slightly improved rate of convergence ($h_\Om$ vs.
$h_\Om\abs{\ln h_\Om}^s$) compared to $h$-FEM but a drastic reduction of the computational
complexity in terms of degrees of freedom from $O(\NO^{1+1/d})$ to $O(\NO(\ln\NO)^2)$.

Note that in the following estimates, we will track the dependence on $Y$ explicitly since as a
last step $Y$ will be chosen $h$-dependent.

\subsection{Graded meshes and \boldmath{$h$}-FEM}\label{sec:lin}

As announced, we are concerned with estimates for \eqref{eq:inty} for the discretization defined in
Section~\ref{sec:lin_def}. For simplicity, in this subsection, we will write $i_yu$ for $i_y^pu$,
since we have $p=(1,1,\dots,1)$ here.

\begin{lemma}[Estimate on $I_1$]\label{lemma:linI1}
  For $\F\in L^2(\Om)$, let  $u\in \mathring{H}^1_L(C,y^\alpha)$ be the solution
  of~\eqref{eq:weakExtension} and let $M\ge h_\Om^{-1}$. Then, it holds
  \[
    \norm{y^{\frac\alpha2}\nabla(u-i_yu)}_{L^2(\Omega\times I_1)}^2\le
    ch_\Omega^{\frac{2s}\mu}Y^{2s}\norm{\F}^2_{L^2(\Omega)}.
  \]
\end{lemma}

\begin{proof}
  First, we observe that the interpolant $i_y\psi_{s,k}$ is constant on $I_1$. By its
  definition~\eqref{eq:def_i}, it holds $(i_y \psi_{s,k})\bigr\rvert_{I_1}=\psi_{s,k}(y_1)$.
  Integration by parts and noting that
  $y^{\alpha+1}(\psi_{s,k}-\psi_{s,k}(y_1))^2\bigr\rvert_0^{y_1}=0$ yields
  \begin{align*}
    \norm{y^{\frac\alpha2}(\psi_{s,k}-\psi_{s,k}(y_1))}_{L^2(I_1)}^2
    &=\int_{I_1}y^\alpha(\psi_{s,k}-\psi_{s,k}(y_1))^2\,dy\\
    &=-\frac2{\alpha+1}\int_{I_1} y^{\alpha+1}(\psi_{s,k}-\psi_{s,k}(y_1))\psi'_{s,k}\,dy\\
    &\le
    \frac2{\alpha+1}\norm{y^{\frac\alpha2}(\psi_{s,k}-\psi_{s,k}(y_1))}_{L^2(I_1)}\norm{y^{\frac\alpha2+1}\psi'_{s,k}}_{L^2(I_1)}.
  \end{align*}
  Then, dividing by $\norm{y^{\frac\alpha2}(\psi_{s,k}-i_y\psi_{s,k})}_{L^2(I_1)}$ implies
  \[
    \norm{y^{\frac\alpha2}(\psi_{s,k}-\psi_{s,k}(y_1))}_{L^2(I_1)}\le c \norm{y^{\frac\alpha2+1}\psi'_{s,k}}_{L^2(I_1)}
  \]
  with $c=\frac1{1-s}$.  By means of Corollary~\ref{cor:psi-2} with $n=1$ and $r=1$, we obtain
  \[
    \norm{y^{\frac\alpha2}(\psi_{s,k}-\psi_{s,k}(y_1))}_{L^2(I_1)}\le c
    \norm{y^{\frac\alpha2+1}\psi_{s,k}'}_{L^2(I_1)}
    \le c\lambda_{k}^{s-\frac12}\norm{y^{s-\frac12}}_{L^2(I_1)}.
  \]
  Hence, we get by Lemma~\ref{lemma:gradedmesh} together with the assumption on $M$ that
  \[
    \lambda_k\norm{y^{\frac\alpha2}(\psi_{s,k}-\psi_{s,k}(y_1))}^2_{L^2(I_1)}\le
    c\lambda_{k}^{2s}h_1^{2s}\le c\lambda_{k}^{2s}h_\Omega^{\frac{2s}\mu}Y^{2s}.
  \]

  In a similar fashion, we obtain by  Corollary~\ref{cor:psi-2} with $n=1$ and $r=0$ the relation
  \[
    \norm{y^{\frac\alpha2}(\psi_{s,k}-\psi_{s,k}(y_1))'}^2_{L^2(I_1)}=\norm{y^{\frac\alpha2}\psi'_{s,k}}^2_{L^2(I_1)}\le
    c\lambda_k^{2s}\norm{y^{s-\frac12}}^2_{L^2(I_1)}\le c\lambda_k^{2s}h_1^{2s}\le
    c\lambda_k^{2s}h_\Omega^{\frac{2s}\mu}Y^{2s},
  \]
  where we again used Lemma \ref{lemma:gradedmesh} in the last step.

  The previous estimates together with \eqref{eq:inty} and Proposition~\ref{prop:SolUF} yield
  \[
    \norm{y^{\frac\alpha2}\nabla(u-i_yu)}_{L^2(\Omega\times I_1)}^2\leq
    ch_\Omega^{\frac{2s}\mu}Y^{2s}\sum_{k=1}^\infty\lambda_k^{2s}\U_k^2
    =ch_\Omega^{\frac{2s}\mu}Y^{2s}\sum_{k=1}^\infty\F_k^2,
  \]
  which implies the assertion.
\end{proof}

\begin{lemma}[Estimates on $I_m$ for $2\le m\le M-1$]\label{lemma:1toN}
  For $\F\in L^2(\Om)$, let $u\in \mathring{H}^1_L(C,y^\alpha)$ be the solution
  of~\eqref{eq:weakExtension}. Moreover, let $M\ge h_\Om^{-1}$ and $\mu\neq s$. Then, it holds
  \[
    \norm{y^{\frac\alpha2}\nabla(u-i_yu)}_{L^2(\Omega\times I_m)}^2\leq
    ch_\Omega^{2}Y^{2\mu}\bigl\{y_m^{2(s-\mu)}-y_{m-1}^{2(s-\mu)}\bigr\}\norm{\F}^2_{L^2(\Omega)}.
  \]
\end{lemma}

\begin{proof}
  For $m\ge2$, we have that $y_{m-1}\le y_m \le 2^{\frac1\mu} y_{m-1}$. It holds
  $i_y\psi_{s,k}=i_1\psi_{s,k}$ such that we conclude with standard estimates for the linear
  Lagrange interpolant $i_1$, Lemma~\ref{lemma:gradedmesh}, and the assumption on $M$ that
  \begin{align*}
    \norm{y^{\frac\alpha2}(\psi_{s,k}-i_1\psi_{s,k})}^2_{L^2(I_m)}
    &\leq c y_{m}^\alpha\norm{\psi_{s,k}-i_1\psi_{s,k}}^2_{L^2(I_m)} \leq cy_{m}^\alpha
    h_m^2\norm{\psi'_{s,k}}^2_{L^2(I_m)} \\
    &\leq cy_{m}^{\alpha+2-2\mu}h_\Omega^2Y^{2\mu}\norm{\psi'_{s,k}}^2_{L^2(I_m)} \leq
    ch_\Omega^2Y^{2\mu}\norm{y^{\frac\alpha2+1-\mu}\psi'_{s,k}}^2_{L^2(I_m)}.
  \end{align*}
  By using Corollary~\ref{cor:psi-2} with $n=1$ and $r=1$, this implies
  \[
    \lambda_k\norm{y^{\frac\alpha2}(\psi_{s,k}-i_1\psi_{s,k})}^2_{L^2(I_m)}
    \le c\lambda_k^{2s}h_\Omega^2Y^{2\mu}\norm{y^{s-\mu-\frac12}}^2_{L^2(I_m)}
    =c\lambda_k^{2s}h_\Omega^2Y^{2\mu}\bigl\{y_m^{2(s-\mu)}-y_{m-1}^{2(s-\mu)}\bigr\}.
  \]

  Similarly, using Corollary~\ref{cor:psi-2} with $n=1$ and $r=0$, we obtain for the term involving
  the derivative
  \begin{align*}
    \norm{y^{\frac\alpha2}(\psi_{s,k}-i_1\psi_{s,k})'}^2_{L^2(I_m)}
    &\leq ch_\Omega^2Y^{2\mu} \norm{y^{\frac\alpha2+1-\mu}\psi''_{s,k}}^2_{L^2(I_m)}
    \le c\lambda_k^{2s}h_\Omega^2Y^{2\mu}\norm{y^{s-\mu-\frac12}}^2_{L^2(I_m)}\\
    &=c\lambda_k^{2s}h_\Omega^2Y^{2\mu}\bigl\{y_m^{2(s-\mu)}-y_{m-1}^{2(s-\mu)}\bigr\}.
  \end{align*}

  The previous estimates in combination with \eqref{eq:inty} yield
  \[
    \norm{y^{\frac\alpha2}\nabla(u-i_yu)}_{L^2(\Omega\times I_m)}^2
    \leq ch_\Omega^2Y^{2\mu}\bigl\{y_m^{2(s-\mu)}-y_{m-1}^{2(s-\mu)}\bigr\}\sum_{k=1}^\infty\lambda_k^{2s}\U_k^2.
  \]
  Finally, applying Proposition~\ref{prop:SolUF}, we get
  \[
    \norm{y^{\frac\alpha2}\nabla(u-i_yu)}_{L^2(\Omega\times I_m)}^2
    \le ch_\Omega^2Y^{2\mu}\bigl\{y_m^{2(s-\mu)}-y_{m-1}^{2(s-\mu)}\bigr\}\sum_{k=1}^\infty\F_k^2,
  \]
  which states the assertion.
\end{proof}

\begin{lemma}[Estimate on $I_M$] \label{lemma:I_M}
  For $\F\in L^2(\Om)$, let $u\in \mathring{H}^1_L(C,y^\alpha)$ be the solution
  of~\eqref{eq:weakExtension}. Moreover, let $2h_\Om^{-1}\ge M\ge h_\Om^{-1}$, $\mu\neq s$, and
  \[
    Y\geq \max\biggl(\frac{3\abs{\ln h_\Omega}}{\sqrt{\lambda_1}},1\biggr).
  \]
  Then, it holds
  \[
    \norm{y^{\frac\alpha2}\nabla(u-i_yu)}_{L^2(\Omega\times I_M)}^2\leq c h_\Omega^2
    \left(Y^{2\mu}\bigl\{Y^{2(s-\mu)}-y_{M-1}^{2(s-\mu)}\bigr\}+1\right)\norm{\F}^2_{L^2(\Om)}.
  \]
\end{lemma}
\begin{proof}
  We recall that $Y=y_M$ and $i_y\psi_{s,k} = \tilde i_1\psi_{s,k}$ on $I_M$. We introduce the
  Lagrange interpolation $i_1$ on $I_M$ as an intermediate function such that
  \begin{align*}
    \norm{y^{\frac\alpha2}(\psi_{s,k}-\tilde i_1\psi_{s,k})}_{L^2(I_M)}
    &\leq
    \norm{y^{\frac\alpha2}(\psi_{s,k}-i_1\psi_{s,k})}_{L^2(I_M)}+\norm{y^{\frac\alpha2}(i_1\psi_{s,k}-\tilde
    i_1\psi_{s,k})}_{L^2(I_M)}\\
    &=
    \norm{y^{\frac\alpha2}(\psi_{s,k}-i_1\psi_{s,k})}_{L^2(I_M)}+\psi_{s,k}(Y)\norm{y^{\frac\alpha2}l_{1,1}}_{L^2(I_M)}\\
    &\leq \norm{y^{\frac\alpha2}(\psi_{s,k}-i_1\psi_{s,k})}_{L^2(I_M)}+cY^{\frac{\alpha+1}2}\psi_{s,k}(Y),
  \end{align*}
  where we used that $\norm{l_{1,1}}_{L^\infty(I_M)}=\norm{\frac{y-y_{M-1}}{h_M}}_{L^\infty(I_M)}=1$ in the
  last step. As in the proof of Lemma \ref{lemma:1toN}, we deduce
  \[
    \lambda_k\norm{y^{\frac\alpha2}(\psi_{s,k}-i_1\psi_{s,k})}^2_{L^2(I_M)}\leq
    c\lambda_k^{2s}h_\Omega^2Y^{2\mu}\bigl\{Y^{2(s-\mu)}-y_{M-1}^{2(s-\mu)}\bigr\}.
  \]
  Since $Y\geq 1$ by assumption, we obtain using Corollary \ref{cor:psi-3} with
  $r_1=\frac{\alpha+1}2=1-s$ together with the monotonicity of $e^{-\sqrt{\lambda_k}y}$
  \[
    Y^{\frac{\alpha+1}2}\psi_{s,k}(Y)\leq c\lambda_k^{\frac{s-1}2}
    e^{-\frac{\sqrt{\lambda_k}}2Y}\leq c\lambda_k^{s-\frac12}
    e^{-\frac{\sqrt{\lambda_1}}2Y},
  \]
  where we notice that $(\lambda_k)_{k\in\N}$ is a non-decreasing sequence. Combining the previous results yields
  \begin{equation}\label{eq:I_M1}
    \lambda_k\norm{y^{\frac\alpha2}(\psi_{s,k}-\tilde i_1\psi_{s,k})}_{L^2(I_M)}^2\leq
    c\lambda_k^{2s}\left(h_\Omega^2Y^{2\mu}\bigl\{Y^{2(s-\mu)}-y_{M-1}^{2(s-\mu)}\bigr\}+
    e^{-\sqrt{\lambda_1}Y}\right).
  \end{equation}
  Similarly, we deduce
  \begin{align*}
    \norm{y^{\frac\alpha2}(\psi_{s,k}-\tilde i_1\psi_{s,k})'}_{L^2(I_M)}
    &\leq
    \norm{y^{\frac\alpha2}(\psi_{s,k}-i_1\psi_{s,k})'}_{L^2(I_M)}+\norm{y^{\frac\alpha2}(i_1\psi_{s,k}-\tilde
    i_1\psi_{s,k})'}_{L^2(I_M)}\\
    &=
    \norm{y^{\frac\alpha2}(\psi_{s,k}-i_1\psi_{s,k})'}_{L^2(I_M)}+\psi_{s,k}(Y)\norm{y^{\frac\alpha2}l_{1,1}'}_{L^2(I_M)}\\
    &\leq
    \norm{y^{\frac\alpha2}(\psi_{s,k}-i_1\psi_{s,k})'}_{L^2(I_M)}+ch_M^{-\frac12}Y^{\frac{\alpha}2}\psi_{s,k}(Y),
  \end{align*}
  where we used that $\norm{l_{1,1}'}_{L^\infty(I_M)}=\norm{h_M^{-1}}_{L^\infty(I_M)}=h_M^{-1}$ in
  the last step. The first term can again be estimated as in the proof Lemma \ref{lemma:1toN} such
  that
  \[
    \norm{y^{\frac\alpha2}(\psi_{s,k}-i_1\psi_{s,k})'}^2_{L^2(I_m)}\leq
    c\lambda_k^{2s}h_\Omega^2Y^{2\mu}\bigl\{Y^{2(s-\mu)}-y_{M-1}^{2(s-\mu)}\bigr\}.
  \]
  Employing Corollary \ref{cor:psi-3} with $r_1=\frac{\alpha}2=\frac12-s$ together with the
  monotonicity of $e^{-\sqrt{\lambda_k}y}$ yields for $Y\geq 1$
  \[
    h_M^{-\frac12}Y^{\frac{\alpha}2}\psi_{s,k}(Y)\leq
    ch_M^{-\frac12}\lambda_k^{\frac{s}2-\frac14} e^{-\frac{\sqrt{\lambda_k}}2Y}\leq
    ch_M^{-\frac12}\lambda_k^{s} e^{-\frac{\sqrt{\lambda_1}}2Y},
  \]
  where we used once again that the sequence $(\lambda_k)_{k\in\N}$ is non-decreasing.
  Due to the previous results, we arrive at
  \begin{equation}\label{eq:I_M2}
    \norm{y^{\frac\alpha2}(\psi_{s,k}-\tilde i_1\psi_{s,k})'}_{L^2(I_M)}^2\leq
    c\lambda_k^{2s}\left(h_\Omega^2Y^{2\mu}\bigl\{Y^{2(s-\mu)}-y_{M-1}^{2(s-\mu)}\bigr\}+h_M^{-1}
    e^{-\sqrt{\lambda_1}Y}\right).
  \end{equation}
  By combining \eqref{eq:inty}, \eqref{eq:I_M1} and \eqref{eq:I_M2}, we obtain
  \begin{align*}
    \norm{y^{\frac\alpha2}\nabla(u-i_yu)}_{L^2(\Omega\times I_M)}^2
    &\leq c\left(h_\Omega^2Y^{2\mu}\bigl\{Y^{2(s-\mu)}-y_{M-1}^{2(s-\mu)}\bigr\}
  +\bigl\{1+h_M^{-1}\bigr\} e^{-\sqrt{\lambda_1}Y}\right)\sum_{k=1}^\infty\lambda_k^{2s}\U_k^2\\
  &= c\left(h_\Omega^2Y^{2\mu}\bigl\{Y^{2(s-\mu)}-y_{M-1}^{2(s-\mu)}\bigr\}
+\bigl\{1+h_M^{-1}\bigr\} e^{-\sqrt{\lambda_1}Y}\right)\norm{\F}^2_{L^2(\Om)},
  \end{align*}
  where we used $\sum_{k=1}^\infty\lambda_k^{2s}\U_k^2=\sum_{k=1}^\infty \F_k^2$ (see Proposition
  \ref{prop:SolUF}) and the definition of $\norm{\F}^2_{L^2(\Om)}$. According to Lemma
  \ref{lemma:gradedmesh}, there holds $h_M^{-1}\leq c Y^{\mu-1} Y^{-\mu}M = c Y^{-1}M$. Since
  \[
    Y\geq \frac{3\abs{\ln h_\Omega}}{\sqrt{\lambda_1}}\quad\text{and}\quad M\le
    2h_\Om^{-1}
  \]
  by assumption, we obtain
  \[
    \bigl\{1+h_M^{-1}\bigr\} e^{-\sqrt{\lambda_1}Y}\leq h_\Omega^3+ch_\Omega^2\abs{\ln
    h_\Omega}^{-1}\leq c h_\Omega^2,
  \]
  which ends the proof.
\end{proof}

\begin{corollary}\label{cor:estY}
  For $\F\in L^2(\Om)$, let $u\in \mathring{H}^1_L(C,y^\alpha)$ be the solution
  of~\eqref{eq:weakExtension}. Moreover, let $2h_\Om^{-1}\ge M\ge h_\Om^{-1}$, $\mu< s$, and
  \[
    2\max\biggl(\frac{3\abs{\ln h_\Omega}}{\sqrt{\lambda_1}},1\biggr)\ge Y\geq \max\biggl(\frac{3\abs{\ln h_\Omega}}{\sqrt{\lambda_1}},1\biggr).
  \]
  Then, it holds
  \[
    \norm{y^{\frac\alpha2}\nabla(u-i_yu)}_{L^2(C_Y)}\leq c h_\Omega\abs{\ln h_
    \Om}^s \norm{\F}_{L^2(\Om)}.
  \]
\end{corollary}
\begin{proof}
  By the Lemmas~\ref{lemma:linI1},~\ref{lemma:1toN}, and~\ref{lemma:I_M}, we obtain
  \begin{align*}
    \norm{y^{\frac\alpha2}\nabla(u-i_yu)}_{L^2(C_Y)}^2
    &=\sum_{m=1}^M\norm{y^{\frac\alpha2}\nabla(u-i_yu)}_{L^2(\Om\times I_m)}^2\\
    &\le \left(ch_\Om^{\frac{2s}\mu}Y^{2s}+ch_\Om^2Y^{2\mu}\sum_{m=2}^M\bigl\{y_m^{2(s-\mu)}-y_{m-1}^{2(s-\mu)}\bigr\}
  +ch_\Om^2\right)\norm{\F}_{L^2(\Om)}^2\\
  &\le ch_\Om^2\left(Y^{2s}-Y^{2\mu} y_1^{2(s-\mu)}+1\right)\norm{\F}_{L^2(\Om)}^2\le
  ch_\Om^2\abs{\ln h_\Om}^{2s}\norm{\F}_{L^2(\Om)}^2,
\end{align*}
where we have used $\mu<s$ and the upper bound on $Y$.
\end{proof}

Now, we are able to state the main result for this subsection analyzing the $h$-FEM on graded
meshes.

\begin{theorem}\label{th:errH}
  For $\F\in \HH^{1-s}(\Om)$, let $\U\in\HH^{s}(\Om)$ and $u\in \mathring{H}^1_L(C,y^\alpha)$ be the
  solutions of~\eqref{eq:Prob} and~\eqref{eq:weakExtension}, respectively, and let $u_h\in
  V_{h,M}$ be the solution of~\eqref{eq:weakdiscreteExtension}. Moreover, let
  $2h_\Om^{-1}\ge M\ge h_\Om^{-1}$, $\mu< s$, and
  \[
    2\max\biggl(\frac{3\abs{\ln h_\Omega}}{\sqrt{\lambda_1}},1\biggr)\ge Y\geq \max\biggl(\frac{3\abs{\ln h_\Omega}}{\sqrt{\lambda_1}},1\biggr).
  \]
  Then, it holds
  \[
    \norm{\U-\trO u_h}_{\HH^s(\Om)}\leq c\norm{y^{\frac\alpha2}\nabla(u-u_h)}_{L^2(C)}\leq c h_\Omega\abs{\ln h_
    \Om}^s \norm{\F}_{\HH^{1-s}(\Om)}.
  \]
\end{theorem}
\begin{proof}
  The first inequality of the assertion is due to Propositions~\ref{prop:trace} and~\ref{prop:equi}.
  Using the Lemmas~\ref{lem:bestapprox} and~\ref{lem:generalint}, we get
  \[
    \norm{y^{\frac\alpha2}\nabla(u-u_h)}_{L^2(C)}\le
    \norm{y^{\frac\alpha2}\nabla(u-\pi_xu)}_{L^2(C_Y)}+\norm{y^{\frac\alpha2}\nabla(u-i_yu)}_{L^2(C_Y)}+\norm{y^{\frac\alpha2}\nabla
    u}_{L^2(C\setminus C_Y)}.
  \]
  The three terms on the right-hand side are estimated in Lemma~\ref{lem:Clement},
  Corollary~\ref{cor:estY}, and Proposition~\ref{prop:truncation}. Hence, we get
  \[
    \norm{y^{\frac\alpha2}\nabla(u-u_h)}_{L^2(C)}\le ch_\Om \norm{\F}_{\HH^{1-s}(\Om)}+c h_\Omega\abs{\ln h_
    \Om}^s \norm{\F}_{L^2(\Om)} + ce^{-\frac{\sqrt{\lambda_1}}2Y}\norm{\F}_{\HH^{-s}(\Om)}.
  \]
  Then, the lower bound on $Y$ yields $e^{-\frac{\sqrt{\lambda_1}}2Y}\le ch_\Om^{\frac32}\le
  ch_\Om$, which implies the assertion.
\end{proof}

\begin{theorem}\label{th:NOY}
  The total number degrees of freedom $\NOY$ in $V_{h,M}$ to achieve the order of convergence given
  in Theorem~\ref{th:errH} behaves like 
  \[
    \NOY=O(\NO^{1+\frac1d}),
  \]
  where $d$ denotes the dimension of $\Om$.
\end{theorem}
\begin{proof}
  For the number of degrees of freedom $\NOY$ of the discretization considered in this section, it holds
  $\NOY = \NO \NY = \NO M$. Then, the assertion follows from $M=O( h_\Om^{-1})=O( \NO^{1/d})$.
\end{proof}

\subsection{Geometric meshes and \boldmath{$hp$}-FEM}\label{sec:p}

In this section, we derive discretization error estimates for the $hp$-method described in
Section~\ref{sec:p_def}, which results in a slightly improved rate of convergence of $h_\Omega$
compared to the previous subsection. However, we will have a drastically reduced computational
complexity in terms of the number of degrees of freedom. To this end, we do neither fix the number
of elements $M$ in direction $y$ nor the slope $\beta$ of the linear degree vector $p$ yet. These
will be set below. As before, we start with estimates for
$\norm{y^{\frac\alpha2}\nabla(u-i_y^pu)}_{L^2(\Omega\times I_m)}^2$ based on~\eqref{eq:inty}.

\begin{lemma}[Estimate on $I_1$]\label{lemma:plinI1}
  For $\F\in L^2(\Om)$, let $u\in \mathring{H}^1_L(C,y^\alpha)$ be the solution
  of~\eqref{eq:weakExtension} and let
  \[
    M \geq \frac{(1+\epsilon)\abs{\ln h_\Omega}}{s\abs{\ln \sigma}}
  \]
  for some $\epsilon\ge0$. Then, it holds
  \[
    \norm{y^{\frac\alpha2}\nabla(u-i^p_yu)}_{L^2(\Omega\times I_1)}^2\leq
    ch_\Omega^{2+2\epsilon}Y^{2s}\norm{\F}^2_{L^2(\Omega)}.
  \]
\end{lemma}

\begin{proof}
  Notice that $i_y^p\psi_{s,k} = \psi_{s,k}(y_1)$ on the first interval $I_1$ as in the previous
  section. Thus, as in the proof of Lemma~\ref{lemma:linI1} but using
  \[
    h_1=\sigma^{M-1}Y\le ch_\Om^{\frac{1+\epsilon}s} Y
  \]
  from Lemma~\ref{lem:GeoMesh_geoMeshProps} and the assumption on $M$, we get
  \begin{align*}
    \lambda_k\norm{y^{\frac\alpha2}(\psi_{s,k}-\psi_{s,k}(y_1))}^2_{L^2(I_1)}&\leq
    c\lambda_{k}^{2s}h_1^{2s}\le c\lambda_{k}^{2s}h_\Omega^{2+2\epsilon}Y^{2s},\\
    \norm{y^{\frac\alpha2}(\psi_{s,k}-\psi_{s,k}(y_1))'}^2_{L^2(I_1)}&\leq
    c\lambda_k^{2s}h_1^{2s}\le c\lambda_k^{2s}h_\Omega^{2+2\epsilon}Y^{2s}.
  \end{align*}
  Hence, we obtain
  \[
    \norm{y^{\frac\alpha2}\nabla(u-i_y^pu)}_{L^2(\Omega\times I_1)}^2\leq
    ch_\Omega^{2+2\epsilon}Y^{2s}\sum_{k=1}^\infty\lambda_k^{2s}\U_k^2.
  \]
  As in the proof of Lemma~\ref{lemma:linI1}, this yields the assertion.
\end{proof}

In order to derive estimates on $I_m$ for $2\le m \le M-1$, we recall the following result which
is a direct consequence of \cite[Lemma 3.2.6]{Melenk2002}.

\begin{proposition}\label{prop:melenk}
  Let $w$ be analytic on $\hat I=(0,1)$ and satisfy for some $c_w, \delta> 0$ the estimate
  \[
    \norm{w^{(n)}}_{L^\infty(\hat I)} \le c_w \delta^n n!\quad\forall n \in \N.
  \]
  Then, there are constants $c, b > 0$ depending only on $\delta$ such that the Gauss-Lobatto
  interpolant $i_q w$ of degree $q\in\N$ on $\hat I$ satisfies
  \[
    \norm{w - i_qw}_{L^\infty(\hat I)} + \norm{(w - i_qw)'}_{L^\infty(\hat I)} \le c c_w e^{-bq}.
  \]
\end{proposition}

\begin{lemma}[Estimates on $I_m$ for $2\le m \le M-1$]\label{lemma:p1toN}
  For $\F\in L^2(\Om)$, let  $u\in \mathring{H}^1_L(C,y^\alpha)$ be the solution
  of~\eqref{eq:weakExtension}. Moreover, let $p\in\N^M$ be a linear degree vector as in
  \eqref{eq:lin_deg_vec} with some $\beta>0$ and let
  \[
    M \geq\frac{ (1+\epsilon)\abs{\ln  h_\Om}}{\min(s,\beta b) \abs{\ln \sigma}}
  \]
  for some $\epsilon\ge0$, where $b>0$ is a constant depending on $\sigma$ only. Then, it holds
  \[
    \norm{y^{\frac\alpha2}\nabla(u-i^p_yu)}_{L^2(\Omega\times I_m)}^2
    \le ch_\Omega^{2+2\epsilon}Y^{2s}\norm{\F}^2_{L^2(\Omega)}.
  \]
\end{lemma}

\begin{proof}
  For $m\ge 2$ it holds $y_{m-1}\le y_m\le \sigma^{-1} y_{m-1}$. By
  transforming to the reference element $\hat I=(0,1)$, we obtain for
  $(i_y^p\psi_{s,k})\bigr\rvert_{I_m} = i_{p_m}\psi_{s,k}\bigr\rvert_{I_m}$ that
  \[
    \norm{y^{\frac\alpha2}(\psi_{s,k}-i_{p_m}\psi_{s,k})}_{L^2(I_m)}
    \le c y_{m-1}^{\frac\alpha2}h_m^{\frac12}\norm{\hat \psi_{s,k}-\hat i_{p_m}\hat
    \psi_{s,k}}_{L^2(\hat I)}
    \le c y_{m-1}^{\frac\alpha2}h_m^{\frac12}\norm{\hat \psi_{s,k}-\hat i_{p_m}\hat
    \psi_{s,k}}_{L^\infty(\hat I)}.
  \]
  By means of Corollary~\ref{cor:psi-2}, applied with $r=1$, and
  Lemma~\ref{lem:GeoMesh_geoMeshProps}, we have
  \[
    \norm{\hat \psi^{(n)}_{s,k}}_{L^\infty(\hat I)}=h_m^n\norm{\psi^{(n)}_{s,k}}_{L^\infty(I_m)}\le
    c\lambda_k^{s-\frac12}h_m^n y_{m-1}^{2s-1-n}8^n  n!\le c
    \lambda_k^{s-\frac12}y_{m-1}^{2s-1}(8(\sigma^{-1}-1))^n n!.
  \]
  Moreover, due to well known series representations of $K_s$ from \cite[9.6.2 and
  9.6.10]{Abramowitz1964}, we directly conclude that $\hat\psi_{s,k}$ is analytic on $\hat I$.
  Hence, Proposition~\ref{prop:melenk} implies with $\delta=8(\sigma^{-1}-1)$ that
  \[
    \norm{\hat \psi_{s,k}-\hat i_{p_m}\hat \psi_{s,k}}_{L^\infty(\hat I)}\le c\lambda_k^{s-\frac12}
    y_{m-1}^{2s-1}e^{-bp_m}.
  \]
  Then, we get by Lemma~\ref{lem:GeoMesh_geoMeshProps}
  \[
    \norm{y^{\frac\alpha2}(\psi_{s,k}-i_{p_m}\psi_{s,k})}_{L^2(I_m)} \le c
    y_{m-1}^{s-\frac12}h_m^{\frac12} \lambda_k^{s-\frac12} e^{-bp_m}\le c
    h_m^s\lambda_k^{s-\frac12}e^{-bp_m}.
  \]

  Analogously, we obtain
  \begin{align*}
    \norm{y^{\frac\alpha2}(\psi_{s,k}-i_{p_m}\psi_{s,k})'}_{L^2(I_m)}
    &\le c y_{m-1}^{\frac\alpha2}h_m^{-\frac12}\norm{(\hat \psi_{s,k}-\hat i_{p_m}\hat
  \psi_{s,k})'}_{L^2(\hat I)}\\
  &\le c y_{m-1}^{\frac\alpha2}h_m^{-\frac12}\norm{(\hat \psi_{s,k}-\hat i_{p_m}\hat
\psi_{s,k})'}_{L^\infty(\hat I)}.
  \end{align*}
  By means of Corollary~\ref{cor:psi-2} with $r=0$ and  Lemma~\ref{lem:GeoMesh_geoMeshProps}, we
  have
  \[
    \norm{\hat \psi^{(n)}_{s,k}}_{L^\infty(\hat I)}=h_m^n\norm{\psi^{(n)}_{s,k}}_{L^\infty(I_m)}\le
    c\lambda_k^{s}h_m^n y_{m-1}^{2s-n}8^n  n!\le c
    \lambda_k^{s}y_{m-1}^{2s}(8(\sigma^{-1}-1))^n n!.
  \]
  Consequently, Proposition~\ref{prop:melenk} yields with $\delta=8(\sigma^{-1}-1)$ that
  \[
    \norm{(\hat \psi_{s,k}-\hat i_{p_m}\hat \psi_{s,k})'}_{L^\infty(\hat I)}\le c\lambda_k^{s}
    y_{m-1}^{2s}e^{-bp_m}.
  \]
  This implies
  \[
    \norm{y^{\frac\alpha2}(\psi_{s,k}-i_{p_m}\psi_{s,k})'}_{L^2(I_m)} \le c
    y_{m-1}^{s+\frac12}h_m^{-\frac12} \lambda_k^{s} e^{-bp_m}\le c \lambda_k^{s}h_m^se^{-bp_m}.
  \]

  Collecting the previous results yields
  \[
    \lambda_k\norm{y^{\frac\alpha2}(\psi_{s,k}-i_{p_m}\psi_{s,k})}_{L^2(I_m)}^2+\norm{y^{\frac\alpha2}(\psi_{s,k}-i_{p_m}\psi_{s,k})'}_{L^2(I_m)}^2\leq
    c \lambda_k^{2s}h_m^{2s}e^{-2bp_m}.
  \]
  Relation~\eqref{eq:lin_deg_vec} implies
  \[
    e^{-2b p_m}\leq c h_1^{2\beta b}h_m^{-2\beta b}.
  \]
  Thus, we deduce
  \begin{equation}\label{eq:p1toN}
    \lambda_k\norm{y^{\frac\alpha2}(\psi_{s,k}-i_{p_m}\psi_{s,k})}_{L^2(I_m)}^2+\norm{y^{\frac\alpha2}(\psi_{s,k}-i_{p_m}\psi_{s,k})'}_{L^2(I_m)}^2
    \leq c\lambda_k^{2s}h_m^{2(s-\beta b)} h_1^{2\beta b}.
  \end{equation}
  Let us now distinguish two cases:
  \begin{itemize}
    \item\textbf{\boldmath{$s\le\beta b$}:} Since $h_m\geq h_1$, we then obtain
      \[
        \lambda_k\norm{y^{\frac\alpha2}(\psi_{s,k}-i_{p_m}\psi_{s,k})}_{L^2(I_m)}^2+\norm{y^{\frac\alpha2}(\psi_{s,k}-i_{p_m}\psi_{s,k})'}_{L^2(I_m)}^2
        \leq c\lambda_k^{2s}h_1^{2(s-\beta b)}
        h_1^{2\beta b} =c\lambda_k^{2s}h_1^{2s}
      \]
      As before, the relation $M\geq\frac{ (1+\epsilon)\abs{\ln h_\Omega}}{s\abs{\ln \sigma}}$ together with
      Lemma~\ref{lem:GeoMesh_geoMeshProps} implies $h_1\le ch_\Om^{\frac{1+\epsilon}s}Y$. Hence, we get
      \[
        \lambda_k\norm{y^{\frac\alpha2}(\psi_{s,k}-i_{p_m}\psi_{s,k})}_{L^2(I_m)}^2+\norm{y^{\frac\alpha2}(\psi_{s,k}-i_{p_m}\psi_{s,k})'}_{L^2(I_m)}^2\leq
        c\lambda_k^{2s}h_\Omega^{2+2\epsilon} Y^{2s}.
      \]
    \item \textbf{\boldmath{$s>\beta b$}:} With $h_m\leq Y$, we get from \eqref{eq:p1toN} that
      \[
        \lambda_k\norm{y^{\frac\alpha2}(\psi_{s,k}-i_{p_m}\psi_{s,k})}_{L^2(I_m)}^2+\norm{y^{\frac\alpha2}(\psi_{s,k}-i_{p_m}\psi_{s,k})'}_{L^2(I_m)}^2\leq
        c\lambda_k^{2s}Y^{2(s-\beta b)} h_1^{2\beta b}.
      \]
      Similarly as before, the relation $M\geq \frac{(1+\epsilon)\abs{\ln h_\Omega}}{\beta b \abs{\ln \sigma}}$
      together with Lemma~\ref{lem:GeoMesh_geoMeshProps} implies $h_1\le ch_\Om^{\frac{1+\epsilon}{\beta b}}Y$. Thus, we get
      \[
        \lambda_k\norm{y^{\frac\alpha2}(\psi_{s,k}-i_{p_m}\psi_{s,k})}_{L^2(I_m)}^2+\norm{y^{\frac\alpha2}(\psi_{s,k}-i_{p_m}\psi_{s,k})'}_{L^2(I_m)}^2\leq
        c\lambda_k^{2s}h_\Omega^{2+2\epsilon} Y^{2s}.
      \]
  \end{itemize}

  The previous results in combination with \eqref{eq:inty} imply
  \[
    \norm{y^{\frac\alpha2}\nabla(u-i_y^pu)}^2_{L^2(\Omega\times I_m)}\leq ch_\Omega^{2+2\epsilon} Y^{2s} \sum_{k=1}^\infty
    \lambda_k^{2s}\U_k^2.
  \]
  Finally, applying Proposition~\ref{prop:SolUF} yields the assertion.
\end{proof}

\begin{lemma}[Estimate on $I_M$]\label{lemma:pI_M}
  For $\F\in L^2(\Om)$, let  $u\in \mathring{H}^1_L(C,y^\alpha)$ be the solution
  of~\eqref{eq:weakExtension}. Moreover, let $p\in\N^M$ be a linear degree vector as in
  \eqref{eq:lin_deg_vec} with some $\beta>0$ and let
  \[
    2\frac{(1+\epsilon)\abs{\ln  h_\Om}}{\min(s,\beta b) \abs{\ln\sigma}} \ge M \geq \frac{(1+\epsilon)\abs{\ln
    h_\Om}}{\min(s,\beta b) \abs{\ln\sigma}}\quad\text{and}\quad
    Y\geq \max\biggl(\frac{2\abs{\ln h_\Omega}}{\sqrt{\lambda_1}},1\biggr)
  \]
  for some $\epsilon\ge0$, where $b>0$ is a constant depending on $\sigma$ only. Then, it holds
  \[
    \norm{y^{\frac\alpha2}\nabla(u-i_y^pu)}_{L^2(\Omega\times I_M)}^2\leq c
    \bigl(h_\Omega^{2+2\epsilon}Y^{2s}+h_\Omega^2\bigr) \norm{\F}^2_{L^2(\Om)}.
  \]
\end{lemma}
\begin{proof}
  We proceed as in the proof of Lemma \ref{lemma:I_M} and recall that $Y=y_M=\sigma^{-1}y_{M-1}$ and
  $h_M=(1-\sigma)Y=(\sigma^{-1}-1)y_{M-1}$. Moreover, according to \cite{Fejer1932}, the Lagrange
  basis functions $l_{i,q}$ of order $q\in\N$ on $I_M$ have the property
  \begin{equation}\label{eq:boundl}
    \norm{l_{i,q}}_{L^\infty(I_M)}\leq 1\quad\text{for }i=0,1,\dots,q.
  \end{equation}
  As a consequence, we obtain by means of an inverse inequality (see, e.g., \cite[Lemma
  3.2.2]{Melenk2002})
  \begin{equation}\label{eq:inversep}
    \norm{l_{i,q}'}_{L^\infty(I_M)}\leq 2q^2\norm{l_{i,q}}_{L^\infty(I_M)}\leq  2q^2.
  \end{equation}
  Noting that $i_y^p\psi_{s,k} = \tilde i_{p_M}\psi_{s,k}$ on $I_M$, we introduce the Gauss-Lobatto
  interpolant $i_{p_M}\psi_{s,k}$ on $I_M$ as an intermediate function such that
  \begin{align*}
    \norm{y^{\frac\alpha2}(\psi_{s,k}-\tilde i_{p_M}\psi_{s,k})}_{L^2(I_M)}
    &\leq
    \norm{y^{\frac\alpha2}(\psi_{s,k}-i_{p_M}\psi_{s,k})}_{L^2(I_M)}+\norm{y^{\frac\alpha2}(i_{p_M}\psi_{s,k}-\tilde
    i_{p_M}\psi_{s,k})}_{L^2(I_M)}\\
    &= \norm{y^{\frac\alpha2}(\psi_{s,k}-i_{p_M}\psi_{s,k})}_{L^2(I_M)}+\psi_{s,k}(Y)\norm{y^{\frac\alpha2}l_{p_M,p_M}}_{L^2(I_M)}\\
    &\leq \norm{y^{\frac\alpha2}(\psi_{s,k}-i_{p_M}\psi_{s,k})}_{L^2(I_M)}+Y^{\frac{\alpha+1}2}\psi_{s,k}(Y),
  \end{align*}
  where we used \eqref{eq:boundl} in the last step. As in the proof of Lemma \ref{lemma:p1toN}, we
  deduce for $M \geq \frac{(1+\epsilon)\abs{\ln  h_\Om}}{\min(s,\beta b) \abs{\ln \sigma}}$ that
  \[
    \lambda_k\norm{y^{\frac\alpha2}(\psi_{s,k}-i_{p_M}\psi_{s,k})}^2_{L^2(I_M)}\leq
    c\lambda_k^{2s}h_\Omega^{2+2\epsilon} Y^{2s}.
  \]
  Since $Y\geq 1$ by assumption, we obtain using Corollary \ref{cor:psi-3} with
  $r_1=\frac{\alpha+1}2=1-s$ together with the monotonicity of $e^{-\sqrt{\lambda_k}y}$
  \[
    Y^{\frac{\alpha+1}2}\psi_{s,k}(Y)\leq c\lambda_k^{\frac{s-1}2}
    e^{-\frac{\sqrt{\lambda_k}}2Y}\leq c\lambda_k^{s-\frac12}
    e^{-\frac{\sqrt{\lambda_1}}2Y},
  \]
  where we notice that $(\lambda_k)_{k\in\N}$ is a non-decreasing sequence. Combining the previous results yields
  \begin{equation}\label{eq:I_M1p}
    \lambda_k\norm{y^{\frac\alpha2}(\psi_{s,k}-\tilde i_{p_M}\psi_{s,k})}_{L^2(I_M)}^2\leq
    c\lambda_k^{2s}\left(h_\Omega^{2+2\epsilon} Y^{2s}+e^{-\sqrt{\lambda_1}Y}\right).
  \end{equation}
  Similarly, we deduce by means of \eqref{eq:inversep}
  \begin{align*}
    \norm{y^{\frac\alpha2}(\psi_{s,k}-\tilde i_{p_M}\psi_{s,k})'}_{L^2(I_M)}
    &\leq
    \norm{y^{\frac\alpha2}(\psi_{s,k}-i_{p_M}\psi_{s,k})'}_{L^2(I_M)}+\norm{y^{\frac\alpha2}(i_{p_M}\psi_{s,k}-\tilde
    i_{p_M}\psi_{s,k})'}_{L^2(I_M)}\\
    &= \norm{y^{\frac\alpha2}(\psi_{s,k}-i_{p_M}\psi_{s,k})'}_{L^2(I_M)}+\psi_{s,k}(Y)\norm{y^{\frac\alpha2}l_{p_M,p_M}'}_{L^2(I_M)}\\
    &\leq \norm{y^{\frac\alpha2}(\psi_{s,k}-i_{p_M}\psi_{s,k})'}_{L^2(I_M)}+2p_M^2Y^{-2}Y^{\frac{\alpha+1}2+2}\psi_{s,k}(Y).
  \end{align*}
  Using $M \geq \frac{(1+\epsilon)\abs{\ln h_\Om}}{\min(s,\beta b)
  \abs{\ln \sigma}}$, the first term can again be estimated as in the proof Lemma \ref{lemma:p1toN}
  such that
  \[
    \norm{y^{\frac\alpha2}(\psi_{s,k}-i_{p_M}\psi_{s,k})'}^2_{L^2(I_m)}\leq
    c\lambda_k^{2s}h_\Omega^{2+2\epsilon} Y^{2s}.
  \]
  Employing Corollary \ref{cor:psi-3} with $r_1=\frac{\alpha+1}2+2=3-s$ together with the
  monotonicity of $e^{-\sqrt{\lambda_k}y}$ yields for $Y\geq 1$
  \[
    Y^{\frac{\alpha+1}2+2}\psi_{s,k}(Y)\leq c\lambda_k^{\frac{s-3}2}
    e^{-\frac{\sqrt{\lambda_k}}2Y}\leq c\lambda_k^{s}e^{-\frac{\sqrt{\lambda_1}}2Y},
  \]
  where we used once again that the sequence $(\lambda_k)_{k\in\N}$ is non-decreasing.  Due to the
  previous results, we arrive at
  \begin{equation}\label{eq:I_M2p}
    \norm{y^{\frac\alpha2}(\psi_{s,k}-\tilde i_{p_M}\psi_{s,k})'}_{L^2(I_M)}^2\leq
    c\lambda_k^{2s}\left(h_\Omega^{2+2\epsilon} Y^{2s}+p_M^4Y^{-4}
    e^{-\sqrt{\lambda_1}Y}\right).
  \end{equation}
  By combining \eqref{eq:inty}, \eqref{eq:I_M1p} and \eqref{eq:I_M2p}, we obtain as in the proof of
  Lemma~\ref{lemma:I_M}
  \begin{align*}
    \norm{y^{\frac\alpha2}\nabla(u-i_y^pu)}_{L^2(\Omega\times I_M)}^2
    &\leq c\left(h_\Omega^{2+2\epsilon} Y^{2s}+\bigl\{1+p_M^4Y^{-4}\bigr\} e^{-\sqrt{\lambda_1}Y}\right)\sum_{k=1}^\infty\lambda_k^{2s}\U_k^2\\
    &= c\left(h_\Omega^{2+2\epsilon} Y^{2s}+\bigl\{1+p_M^4Y^{-4}\bigr\} e^{-\sqrt{\lambda_1}Y}\right)\norm{\F}^2_{L^2(\Om)}.
  \end{align*}
  According to Lemma \ref{lem:GeoMesh_linDegDepOnK}, there holds $p_M\leq cM$. Since
  \[
    Y\geq \frac{2\abs{\ln h_\Omega}}{\sqrt{\lambda_1}}\quad\text{and}\quad M\le \frac{2\abs{\ln  h_\Om}}{\min(s,\beta b) \abs{\ln\sigma}}
  \]
  by assumption, we obtain
  \[
    \bigl\{1+p_M^4Y^{-4}\bigr\} e^{-\sqrt{\lambda_1}Y}\leq
    h_\Omega^2+ch_\Omega^2M^4\abs{\ln
    h_\Omega}^{-4}\leq c h_\Omega^2,
  \]
  which ends the proof.
\end{proof}

\begin{corollary}\label{cor:pestY}
  For $\F\in L^2(\Om)$, let $u\in \mathring{H}^1_L(C,y^\alpha)$ be the solution
  of~\eqref{eq:weakExtension}. Moreover, let $p\in\N^M$ be a linear degree vector as in
  \eqref{eq:lin_deg_vec} with some $\beta>0$ and let
  \[
    2\frac{(1+\epsilon)\abs{\ln  h_\Om}}{\min(s,\beta b) \abs{\ln\sigma}} \ge M \geq \frac{(1+\epsilon)\abs{\ln
    h_\Om}}{\min(s,\beta b) \abs{\ln\sigma}}\quad\text{and}\quad 2\max\biggl(\frac{2\abs{\ln
    h_\Omega}}{\sqrt{\lambda_1}},1\biggr) \ge Y\geq \max\biggl(\frac{2\abs{\ln
    h_\Omega}}{\sqrt{\lambda_1}},1\biggr)
  \]
  for some $\epsilon>0$, where $b>0$ is a constant depending on $\sigma$ only. Then, it holds
  \[
    \norm{y^{\frac\alpha2}\nabla(u-i_y^pu)}_{L^2(C_Y)}\leq c h_\Omega \norm{\F}_{L^2(\Om)}.
  \]
\end{corollary}
\begin{proof}
  By the Lemmas~\ref{lemma:plinI1},~\ref{lemma:p1toN}, and~\ref{lemma:pI_M}, we obtain
  \begin{align*}
    \norm{y^{\frac\alpha2}\nabla(u-i_y^pu)}_{L^2(C_Y)}^2
    &=\sum_{m=1}^M\norm{y^{\frac\alpha2}\nabla(u-i_y^pu)}_{L^2(\Om\times I_m)}^2\le c
    \bigl(h_\Om^{2+2\epsilon}Y^{2s}M +h_\Om^2\bigr)\norm{\F}_{L^2(\Om)}^2\\
    &\le c\bigl(h_\Om^{2+2\epsilon}\abs{\ln h_\Om}^{2s+1}+h_\Om^2\bigr)\norm{\F}_{L^2(\Om)}^2
    \le ch_\Om^2\norm{\F}_{L^2(\Om)}^2,
  \end{align*}
  where we have used the upper bounds on $M$ and $Y$ and the boundedness of $h_\Om^{2\epsilon}\abs{\ln
  h_\Om}^{2s+1}$ for $\epsilon>0$.
\end{proof}

Now, we are able to state the main result for this subsection analyzing the $hp$-FEM on geometric
meshes.

\begin{theorem}\label{th:perrH}
  For $\F\in \HH^{1-s}(\Om)$, let $\U\in\HH^{s}(\Om)$ and $u\in \mathring{H}^1_L(C,y^\alpha)$ be the
  solutions of~\eqref{eq:Prob} and~\eqref{eq:weakExtension}, respectively, and let $u_h\in
  V_{h,M}$ be the solution of~\eqref{eq:weakdiscreteExtension}. Moreover, let $p\in\N^M$
  be a linear degree vector as in \eqref{eq:lin_deg_vec} with some $\beta>0$ and let
  \[
    2\frac{(1+\epsilon)\abs{\ln  h_\Om}}{\min(s,\beta b) \abs{\ln\sigma}} \ge M \geq
    \frac{(1+\epsilon)\abs{\ln
    h_\Om}}{\min(s,\beta b) \abs{\ln\sigma}}\quad\text{and}\quad 2\max\biggl(\frac{2\abs{\ln
    h_\Omega}}{\sqrt{\lambda_1}},1\biggr) \ge Y\geq \max\biggl(\frac{2\abs{\ln
    h_\Omega}}{\sqrt{\lambda_1}},1\biggr)
  \]
  for some $\epsilon>0$, where $b>0$ is a constant depending on $\sigma$ only. Then, it holds
  \[
    \norm{\U-\trO u_h}_{\HH^s(\Om)}\leq c\norm{y^{\frac\alpha2}\nabla(u-u_h)}_{L^2(C)}\leq c h_\Omega\norm{\F}_{\HH^{1-s}(\Om)}.
  \]
\end{theorem}
\begin{proof}
  The first inequality of the assertion is due to Propositions~\ref{prop:trace} and~\ref{prop:equi}.
  Using Lemmas~\ref{lem:bestapprox} and~\ref{lem:generalint}, we get
  \[
    \norm{y^{\frac\alpha2}\nabla(u-u_h)}_{L^2(C)}\le
    \norm{y^{\frac\alpha2}\nabla(u-\pi_xu)}_{L^2(C_Y)}+\norm{y^{\frac\alpha2}\nabla(u-i_yu)}_{L^2(C_Y)}+\norm{y^{\frac\alpha2}\nabla
    u}_{L^2(C\setminus C_Y)}.
  \]
  The three terms on the right-hand side are estimated in Lemma~\ref{lem:Clement},
  Corollary~\ref{cor:pestY}, and Proposition~\ref{prop:truncation}. Hence, we get
  \[
    \norm{y^{\frac\alpha2}\nabla(u-u_h)}_{L^2(C)}\le ch_\Om \norm{\F}_{\HH^{1-s}(\Om)}+c h_\Omega \norm{\F}_{L^2(\Om)} +
    ce^{-\frac{\sqrt{\lambda_1}}2Y}\norm{\F}_{\HH^{-s}(\Om)}.
  \]
  Then, the lower bound on $Y$ yields $e^{-\frac{\sqrt{\lambda_1}}2Y}\le ch_\Om$, which implies the assertion.
\end{proof}

\begin{theorem}\label{th:pNOY}
  The total number degrees of freedom $\NOY$ in $V_{h,M}$ to achieve the order of convergence given in
  Theorem~\ref{th:perrH} behaves like 
  \[
    \NOY=O(\NO(\ln\NO)^2).
  \]
\end{theorem}
\begin{proof}
  As a direct consequence of Lemma~\ref{lem:GeoMesh_dofsInY}, we obtain that the number of degrees
  of freedom $\NOY$ of the discretization considered in this section fulfills $\NOY =\NO\NY =O( \NO
  M^2)$. Then, the assertion follows from $M=O( \abs{\ln h_\Om})=O( \ln \NO)$.
\end{proof}

\section{Numerical Experiments}\label{sec:NumEx}

\subsection{Implementation}

For the discretization with respect to $y$ in $h$-FEM and $hp$-FEM, we use hierarchical Lobatto
polynomials (see, e.g., \cite{solin2003}) as local shape functions on $\hat I=(0,1)$ which then are
transformed onto each interval $I_m$.  Without the weight, i.e., for $s = \frac12$, this would
result in a very sparse structure of the local stiffness matrix, since those shape functions are
orthogonal.  For $s\neq\frac12$ the latter does not hold; nevertheless the global matrix is of
course still sparse.

Let $\eta_i$ for $i = 1,2,\dots,\NO$ be the ansatz functions for the discretization of
$\Omega$ and $\tau_j$ for $j = 1,2,\dots, \NY$ ansatz functions for the discretization of
$(0,Y)$.  On the cylinder $C_Y$ we then use ansatz functions of the form
$\phi_{i,j}(x,y) = \eta_i(x)\tau_j(y)$ with $i = 1,2,\dots,\NO$ and $j =
1,\dots,\NY$.

Due to this special structure, the system matrix $S\in\R^{\NOY\times
\NOY}$ for solving~\eqref{eq:weakdiscreteExtension} can be expressed by means of the
Kronecker product as
\[
  S =B^\text{mass} \otimes A^\text{stiff} + B^\text{stiff} \otimes A^\text{mass}.
\]
Here, $A\in\R^{\NO\times\NO}$ denotes matrices arising from discretization
of $\Om$ and $B\in\R^{\NY\times\NY}$ denotes matrices arising from
discretization of $(0,Y)$ given as
\begin{alignat*}{3}
  A_{ik}^\text{mass}&=\int_\Om \eta_i(x)\eta_k(x)\,dx, &\qquad A_{ik}^\text{stiff}&=\int_\Om
  \nabla\eta_i(x)\nabla\eta_k(x)\,dx,&\qquad i,k&=1,2,\dots,\NO,\\
  B_{jl}^\text{mass}&=\int_0^Y y^\alpha \tau_j(y)\tau_l(y)\,dy, &\qquad B_{jl}^\text{stiff}&=\int_0^Y
  y^\alpha \tau'_j(y)\tau'_l(y)\,dy,&\qquad j,l&=1,2,\dots,\NY.
\end{alignat*}
We observe that one can assemble the matrices $A$ and $B$ completely independent from each other. This
is advantageous since the weight $y^\alpha$ only affects the $B$ matrices, while the $A$ matrices
are standard FEM matrices, which can be computed by any FEM software. 

Using the special structure of $S$, one can implement a memory efficient solvers for the algebraic
systems without ever fully assembling $S$. This will be the topic of a forthcoming paper.

\subsection{Numerical Results}

We take the following configuration from~\cite[Section 6.1]{Nochetto2015}.  For $\Omega = (0,1)^2\subset \R^2$ the
eigenfunctions of the Dirichlet-Laplacian are known to be $\varphi_{k,l}(x) = \sin (k \pi x_1)
\sin(l \pi x_2)$ with corresponding eigenvalues $\lambda_{k,l} = \pi^2(k^2+l^2) $ for
$k,l=1,2,\dots$. For the right-hand side $\F(x) = \lambda_{1,1}^s \varphi_{1,1}(x) =
(2\pi^2)^{s}\sin(\pi x_1)\sin(\pi x_2)$, the solution $\U$ of~\eqref{eq:Prob} and $u$
of~\eqref{eq:weakExtension} are then given by
\[
  \U(x)=\sin(\pi x_1) \sin(\pi x_2)\quad\text{and}\quad
  u(x, y) = \frac{2^{1-\frac s2}\pi^s}{\Gamma(s)} \sin(\pi x_1) \sin(\pi x_2)y^sK_s(\sqrt{2}\pi y).
\]

For the discretization by means of $h$-FEM (cf. the Section~\ref{sec:lin_def} and~\ref{sec:lin}), we
choose the parameters
\[
  \mu = 0.8s,\quad M = \lceil h_\Om^{-1}\rceil,
  \quad\text{and}\quad Y   = \max\biggl(\frac{3\abs{\ln h_\Om}}{\sqrt{2}\pi},1\biggr),
\]
whereas for the discretization by means of $hp$-FEM (cf. the Section~\ref{sec:p_def}
and~\ref{sec:p}), we choose the following parameters:
\[
  \beta = 0.7,\quad \sigma = 0.125,\quad M =  \left\lceil \frac{1.75\abs{\ln h_\Om}}{s \abs{\ln \sigma}}\right\rceil,
  \quad\text{and}\quad Y   = \max\biggl(\frac{3\abs{\ln h_\Om}}{\sqrt{2}\pi},1\biggr).
\]
The orders of convergence stated by the Theorems~\ref{th:errH} and~\ref{th:perrH} in terms of
$h_\Om$ are confirmed by the results of the numerical computations given in Figure~\ref{fig:error}.
Note, that the error $\norm{y^{\frac\alpha2}\nabla(u-u_h)}_{L^2(C)}$ is evaluated by means of the
identity
\[
  \norm{y^{\frac\alpha2}\nabla(u-u_h)}_{L^2(C)}^2=d_s\int_\Om (f\trO u - f\trO u_h)\,dx,
\]
which holds due to the Galerkin orthogonality~\eqref{eq:GalerkinO}.

In Figure~\ref{fig:vergleich}, we depict the errors for both types of discretizations over the total
numbers of degrees of freedom $\NOY$. Thereby, the slower growth of $\NOY$ for the
$hp$-discretization given by Theorem~\ref{th:pNOY} in comparison to Theorem~\ref{th:NOY} clearly
leads to a drastic reduction of the number of degrees of freedom compared to $h$-FEM on a graded
mesh. For instance the number of degrees of freedom to achieve an error of less than $9\cdot
10^{-3}$ in the case $s=0.8$ reduces from $\NOY=1\,072\,692\,225$ for $h$-FEM to $\NOY=9\,661\,477$
for $hp$-FEM, which is a factor of about $111$.

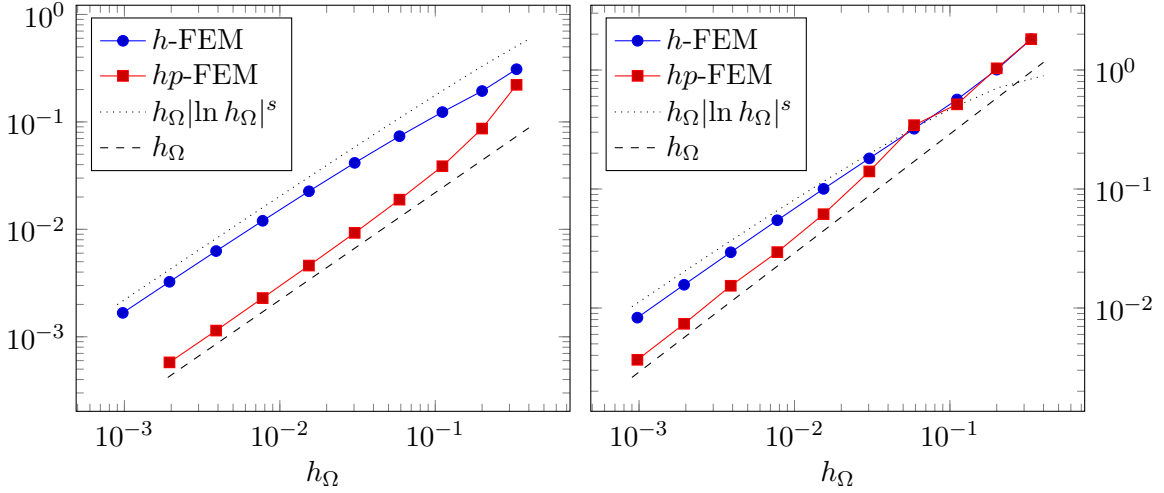
\begin{figure}
  \centering
  \begin{tikzpicture}
    \begin{loglogaxis}
      [
        xlabel=$h_\Om$,
        legend pos=north west,
        width=0.5\textwidth,
        legend cell align={left}
      ]
      \addplot table [y = Error, x  expr = 1/sqrt(\thisrow{NO})]{h_0.2};
      \addlegendentry{$h$-FEM}
      \addplot table [y = Error, x  expr = 1/sqrt(\thisrow{NO})]{hp_0.2};
      \addlegendentry{$hp$-FEM}
      \addplot [dotted, domain=0.0009:0.4, samples=10,]{1.5*x*(-ln(x))^(0.2)};
      \addlegendentry{$h_\Om\abs{\ln h_\Om}^s$}
      \addplot [dashed, domain=0.0019:0.4, samples=10,]{0.22*x};
      \addlegendentry{$h_\Om$}
    \end{loglogaxis}
  \end{tikzpicture}~
  \begin{tikzpicture}
    \begin{loglogaxis}
      [
        xlabel=$h_\Om$,
        legend pos=north west,
        width=0.5\textwidth,
        ylabel near ticks, yticklabel pos=right,
        legend cell align={left}
      ]
      \addplot table [y = Error, x  expr = 1/sqrt(\thisrow{NO})]{h_0.8};
      \addlegendentry{$h$-FEM}
      \addplot table [y = Error, x  expr = 1/sqrt(\thisrow{NO})]{hp_0.8};
      \addlegendentry{$hp$-FEM}
      \addplot [dotted, domain=0.0009:0.4, samples=10,]{2.4*x*(-ln(x))^(0.8)};
      \addlegendentry{$h_\Om\abs{\ln h_\Om}^s$}
      \addplot [dashed, domain=0.0009:0.4, samples=10,]{2.9*x};
      \addlegendentry{$h_\Om$}
    \end{loglogaxis}
  \end{tikzpicture}
  \caption{$\norm{y^{\frac\alpha2}\nabla(u-u_h)}_{L^2(C)}$ for $s=0.2$ (left) and $s=0.8$ (right)
  over $h_\Om$.}\label{fig:error}
\end{figure}

\begin{figure}
  \centering
  \begin{tikzpicture}
    \begin{loglogaxis}
      [
        xlabel=$\NOY$,
        legend pos=south west,
        width=0.5\textwidth,
        legend cell align={left},
        xtick={10, 100, 1000, 10000, 100000, 1000000, 10000000, 100000000, 1000000000}
      ]
      \addplot table [y = Error, x = NOY]{h_0.2};
      \addlegendentry{$h$-FEM}
      \addplot table [y = Error, x = NOY]{hp_0.2};
      \addlegendentry{$hp$-FEM}
    \end{loglogaxis}
  \end{tikzpicture}~
  \begin{tikzpicture}
    \begin{loglogaxis}
      [
        xlabel=$\NOY$,
        legend pos=south west,
        width=0.5\textwidth,
        ylabel near ticks, yticklabel pos=right,
        legend cell align={left},
        xtick={10, 100, 1000, 10000, 100000, 1000000, 10000000, 100000000, 1000000000}
      ]
      \addplot table [y = Error, x = NOY]{h_0.8};
      \addlegendentry{$h$-FEM}
      \addplot table [y = Error, x = NOY]{hp_0.8};
      \addlegendentry{$hp$-FEM}
    \end{loglogaxis}
  \end{tikzpicture}
  \caption{$\norm{y^{\frac\alpha2}\nabla(u-u_h)}_{L^2(C)}$ for $s=0.2$ (left) and $s=0.8$ (right)
  over the total number of degrees of freedom $\NOY$.}\label{fig:vergleich}
\end{figure}
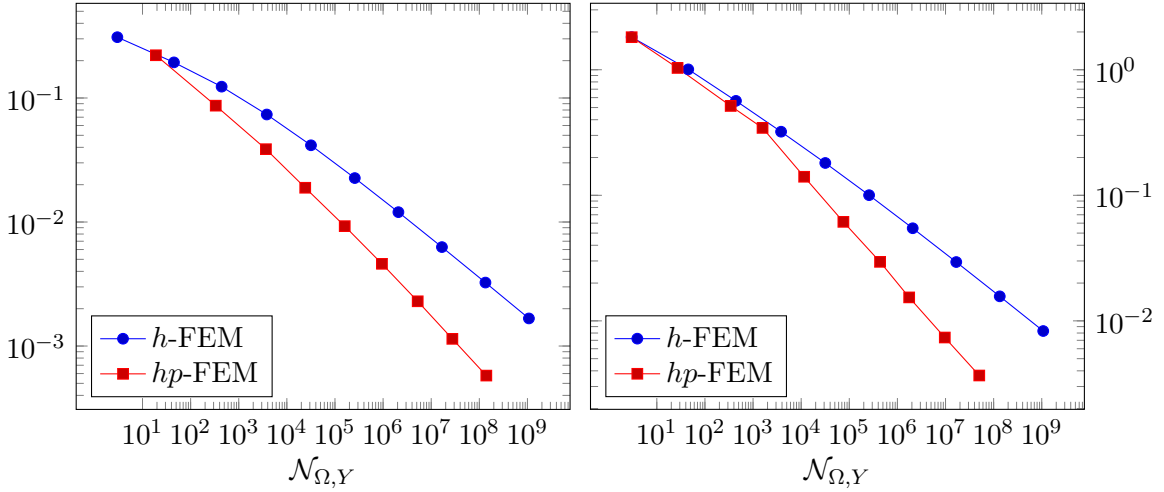

\appendix

\section{Estimates for \boldmath{$\psi_s$} and its derivatives}

We begin with a representation of the derivatives of the expression $z^sK_s(z)$, where $K_s$ are the
modified Bessel functions of second kind. It which will be used in the sequel to derive estimates
for the derivatives of $\psi_s$.

\begin{lemma}\label{lem:SolutionRegularity_derivativeFormula}
  The derivatives of $z^sK_s(z)$ of order $n\in\N_0$ can be calculated as
  \begin{equation}
    (z^{s}K_{s}(z))^{(n)}
    = \sum_{m=0}^n  a_m^n  z^{s-m}K_{s-(n-m)}(z), \label{eq:SolutionRegularity_representation_hypo1}
  \end{equation}
  where the coefficients $a_m^n$ are given by 
  \begin{alignat}{2}
    a_0^n {}&= (-1)^n, \label{eq:SolutionRegularity_representation_hypo2}\\
    a_m^n {}&= (-1)^{n+m}\frac{1}{2^m} \frac{n!}{m!(n-2m)!} &\quad& \text{for } 1\leq m \leq \Bigl\lfloor
    \frac{n}{2} \Bigr\rfloor, \label{eq:SolutionRegularity_representation_hypo3}\\
    a_m^n {}&= 0  &\quad &\text{for } \Bigl\lfloor \frac{n}{2} \Bigr\rfloor<m\le n.  \label{eq:SolutionRegularity_representation_hypo4}
  \end{alignat}
\end{lemma}

\begin{proof}
  We prove this assertion by induction. To this end, we first collect some basic results for the
  modified Bessel function of second kind.  In~\cite[9.6.28]{Abramowitz1964}, we find for all $\nu
  \in\R$
  \[
    \frac{1}{z}\frac{d }{d z}(z^\nu K_\nu(z))= -z^{\nu-1}K_{\nu-1}(z).
  \]
  As a consequence, there holds
  \begin{equation} \label{eq:SolutionRegularity_representation_baseCase}
    ( z^\nu K_\nu(z) )'= -z^\nu K_{\nu-1}(z).
  \end{equation} 
  Using the latter result, we get the following formula for $m,l \in \N_0$:
  \begin{align*}
    (z^{s-m}K_{s-l}(z))'
    &= (z^{l-m} z^{s-l}K_{s-l}(z))'\\
    &=  - z^{l-m} z^{s-l}K_{s-l-1}(z) + (l-m) z^{l-m-1} z^{s-l}K_{s-l}(z)\\
    &= - z^{s-m} K_{s-(l+1)}(z) + (l-m) z^{s-m-1} K_{s-l}(z). 
  \end{align*}
  By means of this, we obtain for $m,n\in \N_0$ with $m\leq n$ by setting $l=n-m\ge 0$
  \begin{equation}
    (z^{s-m} K_{s-(n-m)}(z))' =-z^{s-m}K_{s-(n+1-m)}(z) + (n-2m)z^{s-m-1} K_{s-(n-m)}(z).\label{eq:indu1}
  \end{equation} 
  These elementary results build the basis for the induction: The hypothesis
  \eqref{eq:SolutionRegularity_representation_hypo1} clearly holds for $n=0$.
  
  Assuming that
  \eqref{eq:SolutionRegularity_representation_hypo1}  holds for some $n\in \N_0$, we deduce
  \[
    (z^{s}K_{s}(z))^{(n+1)} = \sum_{m=0}^n  a_m^n( z^{s-m}K_{s-(n-m)}(z))'.
  \]
  Employing \eqref{eq:indu1}, we continue with
  \begin{align*} \label{eq:SolutionRegularity_representation_inductionStep_1}
    (z^{s}K_{s}(z))^{(n+1)}
    &=\sum_{m=0}^n  a_m^n \bigl(-z^{s-m}K_{s-(n+1-m)}(z) + (n-2m)z^{s-m-1} K_{s-(n-m)}(z) \bigr)\\
    &=\sum_{m=0}^n - a_m^n  z^{s-m}K_{s-(n+1-m)}(z)\\
    &\quad+ \sum_{m=1}^{n+1} a_{m-1}^n (n-2(m-1))z^{s-m} K_{s-(n+1-m)}(z)\\
    &=- a_0^n z^{s}K_{s-(n+1)}(z) -n a_{n}^n z^{s-n-1} K_{s}(z)\\
    &\quad+\sum_{m=1}^n \bigl( - a_m^n + a_{m-1}^n (n-2(m-1)) \bigr) z^{s-m}K_{s-(n+1-m)}(z).
  \end{align*}
  It remains to show that
  \[
    a_0^{n+1}=- a_0^n,\qquad a_m^{n+1}=- a_m^n + a_{m-1}^n (n-2(m-1))\text{ for }1\leq m \leq n,\qquad a_{n+1}^{n+1}=-n a_{n}^n.
  \]
  The first and third equation are obvious due to \eqref{eq:SolutionRegularity_representation_hypo2}
  and \eqref{eq:SolutionRegularity_representation_hypo4}. Thus, we only elaborate on the second.
  We distinguish three cases for $1\le m\le n$:
  \begin{itemize}
    \item\textbf{\boldmath{$m\ge\bigl\lfloor\frac{n+1}{2}\bigr\rfloor+2$}:}
      Again due to \eqref{eq:SolutionRegularity_representation_hypo4}, we have $a_m^n =a_{m-1}^n=0$,
      since $m,m-1>\bigl\lfloor\frac{n}{2}\bigr\rfloor$. Hence, it holds
      \[
        - a_m^n + a_{m-1}^n (n-2(m-1))=0=a_{m}^{n+1}.
      \]
    \item\textbf{\boldmath{$m=\bigl\lfloor\frac{n+1}{2}\bigr\rfloor+1$}:} Here, it holds
      $m>\bigl\lfloor\frac{n}{2}\bigr\rfloor$ and we already know that $a_m^n=0$. Moreover, in
      case that $n$ is even, we deduce $m=\frac n2+1$ and
      \[
        n-2(m-1) = n-2\biggl(\frac{n}2+1-1\biggr)= 0.
      \]
      If $n$ is odd, we obtain $a_{m-1}^n =0$ since $m-1=\frac{n+1}{2}>\bigl\lfloor\frac{n}{2}
      \bigr\rfloor$. As a consequence, we get
      \[
        - a_m^n + a_{m-1}^n (n-2(m-1))=0=a_{m}^{n+1}.
      \]
  \item \textbf{\boldmath{$m\le\bigl\lfloor\frac{n+1}{2}\bigr\rfloor$}:} Here, we again distinguish
    between $n$ even and $n$ odd. In the first case, we have $\bigl\lfloor\frac{n+1}{2}
    \bigr\rfloor=\bigl\lfloor\frac{n}2\bigr\rfloor$. Hence, it holds
    $m-1,m\le\bigl\lfloor\frac{n}2\bigr\rfloor$ and by means of
    \eqref{eq:SolutionRegularity_representation_hypo3}, we get
      \begin{align*}
        -a_m^n + (n-2(m-1))a_{m-1}^n
               &= (-1)^{n+m+1} \frac{1}{2^m} \frac{n!}{m!(n-2m)!}\\
               &\quad +(n-2(m-1)) (-1)^{n+m-1} \frac{1}{2^{m-1}} \frac{n!}{(m-1)!(n-2(m-1))!} \\
               &=(-1)^{n+1+m} \frac{1}{2^m} \frac{(n+1)!}{m!(n+1-2m)!} = a_m^{n+1}.
      \end{align*}
      In case that $n$ is odd, we have that $\bigl\lfloor\frac{n+1}{2}\bigr\rfloor=\bigl\lfloor
      \frac{n}{2}\bigr\rfloor+1$. Thus, for $m\leq \bigl\lfloor\frac{n+1}{2}\bigr\rfloor-1$, we can
      reuse the calculations from before. If $m=\bigl\lfloor\frac{n+1}{2}\bigr\rfloor=\frac{n+1}2$,
      we have $a^n_m=0$ such that
      \begin{align*}
        -a_m^n + (n-2(m-1))a_{m-1}^n
        &=a^n_{m-1}=(-1)^{n+\frac{n-1}2}\frac{1}{2^{\frac{n-1}2}} \frac{n!}{\bigl(\frac{n-1}2\bigr)!}\\
        &=(-1)^{n+1+\frac{n+1}2}\frac{1}{2^{\frac{n+1}2}} \frac{(n+1)!}{\bigl(\frac{n+1}2\bigr)!}=a^{n+1}_m.
      \end{align*}
  \end{itemize}
  This ends the proof.
\end{proof}

We next analyze $\psi_s$ defined in~\eqref{eq:psi} and its derivatives with respect to its
boundedness properties.

\begin{lemma}\label{lem:psi-0}
  For $z\ge 0$, it holds
  \[
    0<\psi_s(z) \leq \psi_s(0) = 1.
  \]
\end{lemma}

\begin{proof}
  Since $z^\nu K_\nu(z)>0$ for all $z>0$ and $\nu>-1$, see, e.g.,~\cite[9.6.1]{Abramowitz1964}, and due
  to~\eqref{eq:SolutionRegularity_representation_baseCase}, the function $\psi_s$ is positive and
  monotone decreasing such that $\psi_s(z) \leq \psi_s(0)$ for all $z \ge0$.

  In \cite[9.6.9]{Abramowitz1964} one can find for $\nu>0$ the following behavior of the modified Bessel
  function of the second kind for $z\rightarrow0$:
  \[
    K_\nu(z) \sim\frac{\Gamma(\nu)}{2^{1-\nu}}z^{-\nu}.
  \]
  As a consequence, we obtain by~\eqref{eq:psi}
  \begin{equation}\label{eq:psi=1}
    \lim_{z\rightarrow 0}\psi_s(z)=\lim_{z\rightarrow 0}c_s z^s K_s(z)=1,
  \end{equation}
  which yields together with the foregoing observations the assertion.
\end{proof}

\begin{lemma}\label{lem:psi-2}
  Let $r\in[0,1]$. There exists a constant $c > 0$ depending only on $s$, such that for any $z>0$
  and $n\in\N$ it holds
  \[
    \abs{z^n  \psi_s^{(n)}(z)} \leq c 8^n n! z^{2s-r}.
  \]
\end{lemma}

\begin{proof}
  In order to deduce the bounds for the derivatives of $\psi_s$, we continue with collecting some
  auxiliary results. As before in the proof of Lemma~\ref{lem:psi-0}, we have that $z^rz^\nu
  K_\nu(z)$ is positive for all $z>0$ and $\nu >-1$. Let $\nu_0=\min\bigl(\nu,\frac12\bigr)$. From
  \cite[Theorem 5]{MillerSamko}, we obtain for $z>0$ that $z^{\nu_0}e^zK_\nu(z)$ is a decreasing
  function for all $\nu\ge0$. To employ this, we consider the product
  \[
    z^{\nu+r}K_\nu(z)=z^{\nu+r-\nu_0}e^{-z}\cdot z^{\nu_0}e^{z}K_\nu(z).
  \]
  and note that
  \[
    \operatorname*{arg\,max}_{z\ge 0}z^{\nu + r - \nu_0}e^{-z}=\nu + r -\nu_0\qquad\text{and}\qquad (z^{\nu
    + r - \nu_0}e^{-z})'<0\quad\text{for}\quad z>\nu + r -\nu_0.
  \]
  Hence, $z^{\nu+r}K_\nu(z)$ admits its maximum in the interval $[0,\nu+r-\nu_0]$.  Due to
  Lemma~\ref{lem:psi-0}, we consequently get by~\eqref{eq:psi}
  \begin{equation}\label{eq:boundK}
    z^{\nu+r}K_\nu(z)=\frac{z^r}{c_\nu}\psi_\nu(z)\leq \frac{(\nu + r -
    \nu_0)^r}{c_\nu}\psi_\nu(0)=\frac{(\nu + r - \nu_0)^r}{c_\nu}.
  \end{equation}
  Next, we are concerned with the bounds for the derivatives of $\psi_s$. Employing Lemma
  \ref{lem:SolutionRegularity_derivativeFormula} and the relation $K_\nu(z) = K_{-\nu}(z)$ for
  $\nu\in\R$, see \cite[9.6.6]{Abramowitz1964}, we obtain by~\eqref{eq:psi}
  \[
    \abs{z^{n-2s+r}  \psi_s^{(n)}(z)}=\Abs{z^{n - 2s+r}
      \sum_{m=0}^{\lfloor \frac n2 \rfloor} c_s a_m^n
    z^{s-m}K_{s-(n-m)}(z)}\leq\sum_{m=0}^{\lfloor \frac n2\rfloor}c_s \abs{a^n_m}\abs{z^{n- m
    -s+r} K_{n - m -s}(z)},
  \]
  where the coefficients are given by Lemma \ref{lem:SolutionRegularity_derivativeFormula}. Let
  $\nu(m,n)=n-m-s$. We observe that $\nu(m,n)>0$ since $m\leq \bigl\lfloor \frac n2 \bigr\rfloor$.
  Consequently, \eqref{eq:boundK} yields
  \[
    \abs{z^{n-2s+r}  \psi_s^{(n)}(z)}
    \leq\sum_{m=0}^{\lfloor \frac n2 \rfloor} \frac{c_s}{c_{\nu(m,n)}}\abs{a^n_m} (\nu(m,n) + r -
    \nu_0(m,n))^r
  \]
  with $\nu_0(m,n)=\min\bigl(\nu(m,n),\frac12\bigr)$. Since $n-m\ge1$, it holds 
  \begin{equation}\label{eq:32_4}
    (\nu(m,n) + r - \nu_0(m,n))^r\le (n-m+r)^r\le n-m+1.
  \end{equation}
  Further, using
  \[
    \Gamma(l+\rho)\le \frac{\Gamma(l+1)}{(l+\frac\rho2)^{1-\rho}}
  \]
  from~\cite[estimate (8)]{Lorch1984}, which holds for all $l\in\N_0$ and $0<\rho<1$ , we get by choosing $l=n-m-1$ and $\rho=1-s$
  \begin{equation}\label{eq:32_5}
    \begin{aligned}
      \frac{c_s}{c_{\nu(m,n)}}&=\frac{2^\nu\Gamma(\nu(m,n))}{2^s\Gamma(s)}=\frac1{4^s\Gamma(s)}2^{n-m}\Gamma(\nu(m,n))
      \le\frac1{4^s\Gamma(s)}\frac{2^{n-m}\Gamma(n-m)}{(n-m-1+\frac{1-s}{2})^s}\\
      &\le\frac1{2^s\Gamma(s)(1-s)^s}2^{n-m}(n-m-1)!=c 2^{n-m}(n-m-1)!
    \end{aligned}
  \end{equation}
  with a constant $c$ depending only on $s$.  Using~\eqref{eq:32_4} and~\eqref{eq:32_5}, we get
  \begin{equation}\label{eq:32_2}
    \abs{z^{n-2s+r}  \psi_s^{(n)}(z)}\le c \sum_{m=0}^{\lfloor \frac n2
    \rfloor}2^{n-m}(n-m+1)!\abs{a^n_m}.
  \end{equation}
  Estimating each summand separately, yields by means of Lemma~\ref{lem:SolutionRegularity_derivativeFormula}
  \begin{equation}
    \begin{aligned}
      2^{n-m}(n-m+1)!\abs{a^n_m}
      &= 2^{n-m} (n-m+1)!\frac{1}{2^m} \frac{n!}{m!(n-2m)!}\\\
      &= 2^{n-2m}(m+1)n!\frac{(n-m+1)!}{(m+1)!(n-m+1-(m+1))!}\\
      &=2^{n-2m}(m+1)n!\binom{n-m+1}{m+1}\\
      &\leq 4^n n!\,.\label{eq:32_3}
    \end{aligned}
  \end{equation}
  For the last step, notice that
  \[
    2^{-2m}(m+1)\leq 1\quad\text{and}\quad {{n-m+1}\choose{m+1}}\leq 2^n.
  \]
  Finally, \eqref{eq:32_2} and \eqref{eq:32_3} yield the assertion since $\bigl\lfloor \frac n2
  \bigr\rfloor+1\leq 2^n$.
\end{proof}

Finally, we state a result about the exponential decay of $\psi_s$ and its derivative.
\begin{lemma}\label{lem:psi-3}
  The following assertions hold:
  \begin{enumerate}[(a)]
    \item\label{it:1} Let $s_0=\min(s,\frac12)$. Moreover, let $z\geq a>0$ and $r\geq s_0-s$. Then
      there exists a constant $c$ only depending on $r$ and $s$ such that
      \[
        \abs{z^{r}\psi_s(z)}\leq c e^{-\frac{z}2}\cdot
        c_sa^{s_0}e^aK_s(a).
      \]
    \item\label{it:2} Let $s_0'=\min(1-s,\frac12)$. Moreover, let $z\geq a>0$ and $r\geq s_0'-s$.
      Then there exists a constant $c$ only depending on $r$ and $s$ such that
      \[
        \abs{z^{r}\psi_s'(z)}\leq c e^{-\frac{z}2}\cdot
        c_{s}a^{s_0'}e^aK_{1-s}(a).
      \]
  \end{enumerate}
\end{lemma}
\begin{proof}
  We start as in the proof of Lemma \ref{lem:psi-2}. According to \cite[Theorem 5]{MillerSamko}, we
  get for $z\geq0$ that $z^{s_0}e^zK_s(z)$ is a decreasing function. Consequently, having in mind
  the definition of $\psi_s$ and that $z^\nu K_\nu(z)>0$ for all $z\geq0$ and $\nu>-1$, see the
  proof of Lemma \ref{lem:psi-0}, we obtain
  \[
    \abs{z^{r} \psi_s(z)}=z^{r+s-s_0}e^{-z}\cdot c_s z^{s_0}e^zK_s(z)\leq z^{r+s-s_0}e^{-z}\cdot c_sa^{s_0}e^aK_s(a).
  \]
  This is already the desired result for $r=s_0-s$ noticing that $e^{-\frac{z}{2}}<1=c$ for $z\geq
  0$. For $r>s_0-s$, we observe that
  \[
    z^{r+s-s_0}e^{-z}=z^{r+s-s_0}e^{-\frac{z}2}e^{-\frac{z}2}\leq c e^{-\frac{z}2},
  \]
  where we used that
  \[
    \max_{z\geq 0}(z^{r+s-s_0} e^{-\frac{z}2})=\left(\frac{2(r+s-s_0)}{e}\right)^{r+s-s_0}=c.
  \]
  Combining the previous results yields the first inequality of the assertion. Next, we deduce by
  means of the definition of $\psi_s$ and \eqref{eq:SolutionRegularity_representation_baseCase}
  \[
    z^{r} \psi_s'(z)=\frac{c_s}{c_{1-s}}z^{r+2s-1}c_{1-s}z^{1-s}K_{1-s}(z)=\frac{c_s}{c_{1-s}}z^{r+2s-1}\psi_{1-s}(z)
  \]
  such that the second inequality of the assertion follows from the first one noting that
  $r+2s-1\ge s_0'-s+2s-1=\min(1-s,\frac12)-(1-s)$ by the assumption on $r$.
\end{proof}

\section*{Acknowledgement}

The authors acknowledge Christian Kahle's and Martin Stoll's support in the efficient solution of
the arising linear systems.

\bibliographystyle{abbrv}
\bibliography{lit}

\end{document}